\documentclass[11pt]{amsart}
\usepackage{graphicx} 
\usepackage{mathtools}
\usepackage{amsthm,amssymb}
\usepackage[usenames,dvipsnames]{xcolor}
\usepackage{fullpage}
\usepackage{tikz,hyperref}
\usetikzlibrary{cd}
\usepackage{enumerate}
\usepackage{combelow}
\usepackage{comment}
\usepackage{youngtab}
\usepackage{ytableau}

\newcommand{\SG}[1]{{\color{red}#1--Sean}}
\newcommand{\EG}[1]{{\color{blue}#1--Eugene}}
\newcommand{\MG}[1]{{\color{ForestGreen}#1--Maria}}

\theoremstyle{theorem}
\newtheorem{thm}{Theorem}[section]
\newtheorem{prop}[thm]{Proposition}
\newtheorem{lem}[thm]{Lemma}
\newtheorem{exa}[thm]{Example}
\newtheorem{rmk}[thm]{Remark}
\newtheorem{cor}[thm]{Corollary}

\theoremstyle{definition}
\newtheorem{defn}[thm]{Definition}
\newtheorem{example}[thm]{Example}
\newtheorem{problem}[thm]{Problem}

\newcommand{\bC}{\mathbb{C}}
\newcommand{\bZ}{\mathbb{Z}}
\newcommand{\bK}{\mathbb{K}}
\newcommand{\cO}{\mathcal{O}}
\newcommand{\bQ}{\mathbb{Q}}
\newcommand{\bP}{\mathbf{P}}
\newcommand{\bD}{\mathbb{D}}
\newcommand{\inv}{\mathrm{inv}}
\newcommand{\codim}{\mathrm{codim}}

\newcommand{\area}{\mathrm{area}}
\newcommand{\dinv}{\mathrm{dinv}}

\newcommand{\sweep}{\mathrm{sweep}}
\newcommand{\rank}{\mathrm{rk}}
\newcommand{\PF}{\mathrm{PF}}
\newcommand{\arm}{\mathrm{arm}}
\newcommand{\leg}{\mathrm{leg}}
\newcommand{\tdinv}{\mathrm{tdinv}}
\newcommand{\maxtdinv}{\mathrm{maxtdinv}}

\newcommand{\hdinv}{\mathrm{hdinv}}
\newcommand{\pathdinv}{\mathrm{pathdinv}}
\newcommand{\WPF}{\mathrm{WPF}}
\newcommand{\wWPF}{\overline{\mathrm{WPF}}}

\newcommand{\RD}{\mathrm{RD}}
\newcommand{\LD}{\mathcal{LD}}
\newcommand{\Stack}{\mathrm{Stack}}
\newcommand{\stack}{\mathrm{stack}}

\newcommand{\grFrob}{\mathrm{grFrob}}
\newcommand{\Hilb}{\mathrm{Hilb}}
\newcommand{\rev}{\mathrm{rev}}

\newcommand{\Frob}{\mathrm{Frob}}

\newcommand{\Sp}{\mathrm{Sp}}
\newcommand{\Gr}{\mathrm{Gr}}
\newcommand{\Ad}{\mathrm{Ad}}
\newcommand{\Lie}{\mathrm{Lie}}
\newcommand{\JT}{\mathrm{JT}}
\newcommand{\IC}{\mathrm{IC}}

\newcommand{\GL}{\mathrm{GL}}
\newcommand{\ev}{\mathrm{ev}}
\newcommand{\pr}{\mathrm{pr}}
\newcommand{\BM}{\mathrm{BM}}
\newcommand{\diag}{\mathrm{diag}}

\newcommand{\aGr}{\widetilde{\mathrm{Gr}}}
\newcommand{\aFl}{\widetilde{\mathrm{Fl}}}

\newcommand{\ttt}{\mathbf{t}}
\newcommand{\ww}{\mathbf{w}}
\newcommand{\bO}{\mathbb{O}}

\newcommand{\std}{\mathrm{std}}

\newcommand{\End}{\mathrm{End}}

\newcommand{\cE}{\mathcal{E}}

\newcommand{\rslocus}{\Xi}

\title{A geometric interpretation for the Delta Conjecture}
\author{Maria Gillespie} 
\address{Department of Mathematics, Colorado State University, Fort Collins, CO 80523, USA}
\email{maria.gillespie@colostate.edu}
\thanks{The first author was partially supported by NSF DMS award number 2054391.}

\author{Eugene Gorsky}
\address{Department of Mathematics, University of California Davis\\ One Shields Ave, Davis CA 95616, USA}
\email{egorskiy@math.ucdavis.edu}
\thanks{The second author was partially supported by NSF DMS award number 2302305.}

\author{Sean T. Griffin}
\address{Faculty of Mathematics, University of Vienna, Oskar-Morgenstern-Platz 1, 1090 Vienna, Austria}
\email{sean.griffin@univie.ac.at}
\thanks{The third author was supported by ERC grant “Refined invariants in combinatorics, low-dimensional topology and geometry of moduli spaces” No.~101001159}

\date{\today}

\begin{document}

\maketitle
\begin{abstract}
    We introduce a variety $Y_{n,k}$, which we call the \textit{affine $\Delta$-Springer fiber}, generalizing the affine Springer fiber studied by Hikita, whose Borel-Moore homology has an $S_n$ action and a bigrading that corresponds to the Delta Conjecture symmetric function $\mathrm{rev}_q\,\omega \Delta'_{e_{k-1}}e_n$ under the Frobenius character map.  We similarly provide a geometric interpretation for the Rational Shuffle Theorem in the integer slope case $(km,k)$. The variety $Y_{n,k}$ has a map to the affine Grassmannian whose fibers are the $\Delta$-Springer fibers introduced by Levinson, Woo, and the third author. Part of our proof of our geometric realization relies on our previous work on a Schur skewing operator formula relating the Rational Shuffle Theorem to the Delta Conjecture.
\end{abstract}

\tableofcontents

\section{Introduction}

In this paper, we give geometric realizations of both the Rectangular Shuffle Theorem in the $(mk,k)$ case and for the Delta Conjecture, in terms of affine Springer fibers.  This is the second of a two-part series of papers starting with \cite{GGG1}, which both algebraically and combinatorially establishes a skewing formula relating the polynomials from the $(mk,k)$ Rectangular Shuffle Theorem and those of the Delta Conjecture, building from the work of \cite{BHMPS}.  This skewing formula is key to the proof of our geometric construction.



We summarize the results of the paper in the following table:
\begin{center}
\begin{tabular}{|c|c|c|c|c|}
\hline
Degree & Algebra & Combinatorics & Geometry & Module\\
\hline
$K$ & $E_{K,k}\cdot 1$ & $\PF_{K,k}$ & $X_{n,k}$ & $H_*^{BM}(X_{n,k})\circlearrowleft S_K$\\
\hline
$n$ & $\Delta'_{e_{k-1}}e_n$ & $\LD_{n,k}^{\stack}$ & $Y_{n,k}$ & $H_*^{BM}(Y_{n,k})\circlearrowleft S_n$\\
\hline
\end{tabular}
\end{center}
We explain various objects in the table and relations between them in the following sections. In particular, the first row corresponds to algebraic, combinatorial and geometric avatars of a certain degree $K=k(n-k+1)$ symmetric function, while the second row corresponds to a degree $n$ symmetric function. The two symmetric functions are related by an explicit skewing operator $s_{\lambda}^{\perp}$, see Theorem \ref{thm:IntroSkewing}.


\subsection{Shuffle Theorems}

A \textbf{labeled $(n,n)$ Dyck path} or \textbf{parking function} is a Dyck path in the $n\times n$ grid whose vertical steps are labeled with positive integers, such that the labeling increases up each vertical run. See the left-most example in Figure~\ref{fig:PF-example}. The Shuffle Theorem~\cite{CM} gives the following remarkable combinatorial formula for the evaluation $\nabla e_n$,
\begin{equation}
\label{eq: nabla intro}
\nabla e_n = \sum_{P \in \WPF_{n,n}} t^{\area(P)} q^{\dinv(P)} x^P.
\end{equation}
See Section \ref{sec: background} for relevant definitions. Similarly, given $\gcd(a,b)=1$ one can define a {\bf $(ka,kb)$ rational parking function} as a lattice path (also known as a rational Dyck path) in the $(ka)\times (kb)$ grid that stays weakly above the line $y=ax/b$, starts in the southwest corner, and ends in the northeast corner, together with a column-strictly-increasing labeling of the up steps (see the middle path in Figure~\ref{fig:PF-example} for an example where $k=3$, $a=2$ and $b=1$). The combinatorial statistics $\area$ and $\dinv$ in the ``combinatorial" right hand side of \eqref{eq: nabla intro} has a natural generalization to rational parking functions. 

However, to generalize the ``algebraic'' left hand side one needs to consider the {\bf Elliptic Hall Algebra} $\cE_{q,t}$ defined in \cite{EHA} and extensively studied in the last two decades \cite{RationalShuffle,BHMPS,BHMPS2,GN,SV1,SV2,Negut}. This is a remarkable algebra acting on the space $\Lambda(q,t)$ of symmetric functions in infinitely many variables with coefficients in $\bQ(q,t)$. The Rational Shuffle Theorem, 
proposed by Bergeron, Garsia, Leven and Xin~\cite{RationalShuffle} (see also \cite{GN} for $k=1$ case and a connection to Hilbert schemes), subsequently proven by Mellit~\cite{Mellit}, states that
\[
E_{ka,kb} \cdot (-1)^{k(a+1)} = \sum_{P\in \PF_{ka,kb}} t^{\area(P)} q^{\dinv(P)} x^P,
\]
where $E_{ka,kb}$ is a particular element of  $\cE_{q,t}$.
Note that our conventions for $a$ and $b$ are flipped from~\cite{RationalShuffle}. 

We will be primarily interested in the case $(ka,kb)=(K,k)$ where $K=k(n-k+1)$ (so that $a=n-k+1$ and $b=1$)  and consider parking functions in the $K\times k$ rectangle. The corresponding symmetric function $E_{K,k}\cdot 1$ has degree $K$.

\begin{rmk}
For $k=n$ we have $K=k$, and the symmetric function in question is $E_{n,n}\cdot 1=\nabla e_n$.
\end{rmk}

\subsection{Affine Springer fibers}
In a different line of work, Hikita \cite{Hikita} gave an alternative geometric interpretation for $\nabla e_n$ in terms of {\bf affine Springer fibers}. Given a nil-elliptic operator $\gamma\in \mathfrak{sl}_n[[\epsilon]]$, one can define the affine Springer fiber $\Sp_{\gamma}$ in the affine flag variety $\aFl$. The group $S_n$ has a Springer-like representation in the (Borel-Moore) homology of $\Sp_{\gamma}$. Affine Springer fibers and their homology have been a subject of very active study in geometric representation theory, see \cite{Yun} and references therein. In particular, $\Sp_{\gamma}$ is closely related to compactified Jacobians and Hilbert schemes of points on plane curve singularities \cite{MY,MS,MaulikShen} and serves as a local model for fibers in the Hitchin integrable system \cite{OY,OY2}, while 
the point count of $\Sp_{\gamma}$ over a finite field is related to orbital integrals in number theory \cite{KivTsai}. The conjectures of Oblomkov, Rasmussen and Shende \cite{ORS} relate the homology of affine Springer fibers to Khovanov-Rozansky link homology.

If $\gamma$ is regular and semisimple, the geometry of $\Sp_{\gamma}$ is controlled by the characteristic polynomial $f_{\gamma}(z,\epsilon)=\det(\gamma-z I)$ and the {\bf spectral curve} $\{f_{\gamma}(z,\epsilon)=0\}$.  Hikita considered in \cite{Hikita} the case when 
$$
f_{\gamma}(z,\epsilon)=z^n-\epsilon^{n+1}
$$
and proved that $\Sp_{\gamma}$ admits an affine paving with cells in bijection with $(n,n)$ parking functions and the (Frobenius) character of the Borel-Moore homology of $\Sp_{\gamma}$ (as a bigraded $\bQ S_n$-module) matches the right hand side of \eqref{eq: nabla intro} (up to a minor twist). In particular, the dimension of the cells is closely related to the $\dinv$ statistics on parking functions. In \cite{CO} Carlsson and Oblomkov used this work to construct a long-sought explicit basis in the space of diagonal coinvariants.

In \cite{GMV} the second author, Mazin and Vazirani generalized the results of \cite{Hikita} to $(a,b)$-parking functions with $\gcd(a,b)=1$ (and $k=1$ in the above notations): for 
$$
f_{\gamma}(z,\epsilon)=z^a-\epsilon^{b}
$$
the affine Springer fiber also admits an affine paving with cells in bijection with $(a,b)$ parking functions. In particular, \cite{GMV} clarified the relation between $(a,b)$ parking functions and a certain subset of (extended) affine permutations in $\widetilde{S_a}$.

In \cite{GMO} the second author, Mazin and Oblomkov made progress towards the general non-coprime case $(ka,kb)$ by considering a more complicated class of ``generic'' spectral curves $\{f_{\gamma}(z,\epsilon)\}$. They proved that for such $\gamma$ the affine Springer fiber $\Gr_{\gamma}$ in the affine Grassmannian admits an affine paving with cells labeled by $(ka,kb)$-Dyck paths. It is plausible that the affine Springer fiber $\Sp_{\gamma}$ also admits an affine paving with cells labeled by $(ka,kb)$ parking functions, but we do not need it here.

Instead, we propose a {\bf different} geometric model for a special subclass of the non-coprime case. Let $K=k(n-k+1)$, we consider a family of nil-elliptic elements $\gamma_{n,k,N}\in \mathfrak{sl}_K[[\epsilon]]$ with characteristic polynomials $z^K-\epsilon^{k+N}$.
\begin{defn}\label{def:IntroVarieties}
    Let $\gamma=\gamma_{n,k,N}$ be the $\mathcal{O}$-linear operator on $\mathcal{O}^K = \mathbb{C}^K[[\epsilon]]$ defined by
    \[
    \gamma e_i = 
    \begin{cases}
    e_{i+k} & \text{if }1\leq i\leq (n-k)k\\
    e_{i+k+1} & \text{if } (n-k)k < i < K\\
    \epsilon^{N+1} e_1 & \text{if }i=K.
    \end{cases}
    \]
\end{defn}

Here we extend the basis periodically by $e_{i+K}=\epsilon e_i$. We also define a certain union $C$ of affine Schubert cells and the subvariety 
$$
X_{n,k,N}=C\cap \Sp_{\gamma}.
$$
We can now state our first main result:

\begin{thm}
\label{thm: intro Y}~
\begin{enumerate}[(a)]
\item For $N\ge k$ the space $X_{n,k}=X_{n,k,N}$ does not depend on $N$
and admits an affine paving in which the cells are in bijection with $(K,k)$ parking functions. 
\item For all $N$, the Borel-Moore homology of $X_{n,k,N}$ has an action of $S_K$. For $N\ge k$ the Frobenius character of this action equals 
$$
\mathrm{Frob}(H_*^{BM}(X_{n,k});q,t) = \mathrm{rev}_q\,\omega (E_{K,k}\cdot 1).
$$ 
where the $q$ parameter keeps track of homological degree and the $t$ grading keeps track of the connected  component of $\aFl$.
\end{enumerate}
\end{thm}

It would be interesting to find the analogues of $X_{n,k}$
for more the more general family of symmetric functions $E_{km,kn}\cdot 1$ where $m$ and $n$ are coprime.

\begin{rmk}
For $k=n$ we have $\gamma_{n,n,N}\,e_i=e_{i+n+1}$ for $i<n$, and $\gamma_{n,n,N}\,e_K=\epsilon^{N+1}e_1$. For $N=n$, we recover Hikita's result but with a slightly different variety. In particular, $X_{n,n,n}$ is a non-compact subvariety of the affine partial flag variety $\aFl_{(1^n)} = \GL_n(\bC((\epsilon)))/I^-$ associated to $\GL$, and the space $X_{n,n,n}$ has one connected component for each possible value of the area statistic on an $n\times n$ Dyck path. On the other hand, Hikita studied the $(n,n+1)$ affine Springer fiber associated to the operator $\gamma_{n,n,0}$ inside of $\mathrm{SL}_n(\bC((\epsilon)))/I$, which is a compact subvariety of the affine flag variety associated to $\mathrm{SL}$. The content of Theorem~\ref{thm: intro Y} in the case $k=n$ then says that the affine Springer fiber studied by Hikita has the same Borel--Moore homology as that of $X_{n,n,n}$, even though the spaces are not isomorphic.


\end{rmk}



\subsection{Delta Conjecture}

A second generalization of the Shuffle Theorem called the {\bf Delta Conjecture} was formulated by Haglund, Remmel, and Wilson \cite{HRW}. It involves a more general Macdonald eigenoperator $\Delta'_{e_{k-1}}$ and relates it to \emph{stacked parking functions}. See the image on the right of Figure~\ref{fig:map-F} for an example of a $(n,k)$ stacked parking function for $n=6$ and $k=3$.  The (Rise) Delta Conjecture (reformulated here in terms of stacked parking functions) states 
\[
\Delta'_{e_{k-1}}e_n = \sum_{P\in \LD^{\mathrm{stack}}_{n,k}} t^{\area(P)}q^{\hdinv(P)} x^P.
\]
This version of the Delta Conjecture was proven by~\cite{DAdderio} and independently by~\cite{BHMPS}. To the authors' knowledge, the Valley version (involving a statistic $\mathrm{wdinv}$) of the conjecture remains open.

In~\cite{GGG1}, we proved a formula that directly relates the Rational Shuffle Theorem to the Delta Conjecture. 
\begin{thm}[{\cite[Theorem 1.1]{GGG1}}]\label{thm:IntroSkewing}
    Letting $K = k(n-k+1)$ and $\lambda = (k-1)^{n-k}$, we have
    \begin{equation}\label{eq:IntroSkewingEqn}
    \Delta'_{e_{k-1}}e_n = s_{\lambda}^\perp (E_{K,k}\cdot 1),
    \end{equation}
    where $s_\lambda^\perp$ is the adjoint to multiplication by the Schur function $s_\lambda$.
\end{thm}

The proof uses the relation between the Elliptic Hall Algebra and the Shuffle algebra and certain identities for $E_{K,k}$ studied in  \cite{BHMPS,BHMPS2,Negut}.

Our main result is a geometric version of \eqref{eq:IntroSkewingEqn} by constructing a family of subvarieties $Y_{n,k,N}$ in the partial affine flag variety:
$$
Y_{n,k,N} \coloneqq C'\cap \{\Lambda_\bullet \in \aFl_{(K-n,1^n)} \mid \gamma \Lambda_i\subseteq \Lambda_i\,\forall i,\, \JT(\gamma|_{\Lambda_0/\Lambda_{K-n}}) \leq (n-k)^{k-1}\},
$$
where $\leq$ above is \emph{dominance order} on partitions.
Equivalently, the Jordan type condition above may be written $\gamma^{n-k}\Lambda_0\subseteq \Lambda_{K-n}$. 
See Definition~\ref{def:Varieties} for more details, here we identify the affine partial flag variety with the space of flags of lattices, and $\JT(\gamma|_{\Lambda_0/\Lambda_{K-n}}) $ refers to the Jordan type of the induced action of $\gamma=\gamma_{n,k,N}$ on the quotient $\Lambda_0/\Lambda_{K-n}$.


\begin{thm}\label{thm:GradedFrob}~
\begin{enumerate}[(a)]
\item For all $N\geq k$, the space $Y_{n,k}=Y_{n,k,N}$ does not depend on $N$, and its Borel-Moore homology admits an action of $S_n$ such that
$$
q^{\binom{k-1}{2}(n-k)}\mathrm{Frob}(H_*^{BM}(Y_{n,k});q,t)=s_{\lambda'}^{\perp}\,\mathrm{Frob}(H_*^{BM}(X_{n,k});q,t)
$$
where $\lambda' = (n-k)^{k-1}$.
\item We have 
    $$
    \mathrm{Frob}\left(H_*^{BM}(Y_{n,k});q,t\right) = \mathrm{rev}_q\,\omega (\Delta'_{e_{k-1}}e_n),
    $$
    where the $q$ parameter keeps track of homological degree and the $t$ grading keeps track of the connected  component of $\widetilde{Fl}_{(K-n,1^n)}$.
\end{enumerate}
\end{thm}

In fact, part (b) follows from Theorem \ref{thm:IntroSkewing} and part (a). Part (a) is proved similarly to the main result of \cite{GG}, see below.

\begin{rmk} In the case $n=k$, Theorem~\ref{thm:GradedFrob} part (a) is trivial. Indeed, when $n=k$ we get $X_{n,n,N}=Y_{n,n,N}$ since $\Lambda_n=\Lambda_K$ and the Jordan type condition is vacuous. On the other hand, $\lambda' = \emptyset$ and
$$
\Delta_{e'_{n-1}}e_n=\nabla e_n=E_{n,n}\cdot 1.
$$
\end{rmk}

Lastly, we use Theorem~\ref{thm:IntroSkewing} and the Rational Shuffle Theorem to give a new combinatorial proof of the (Rise) Delta Conjecture.



\begin{figure}
    \centering
    \includegraphics[scale=0.5]{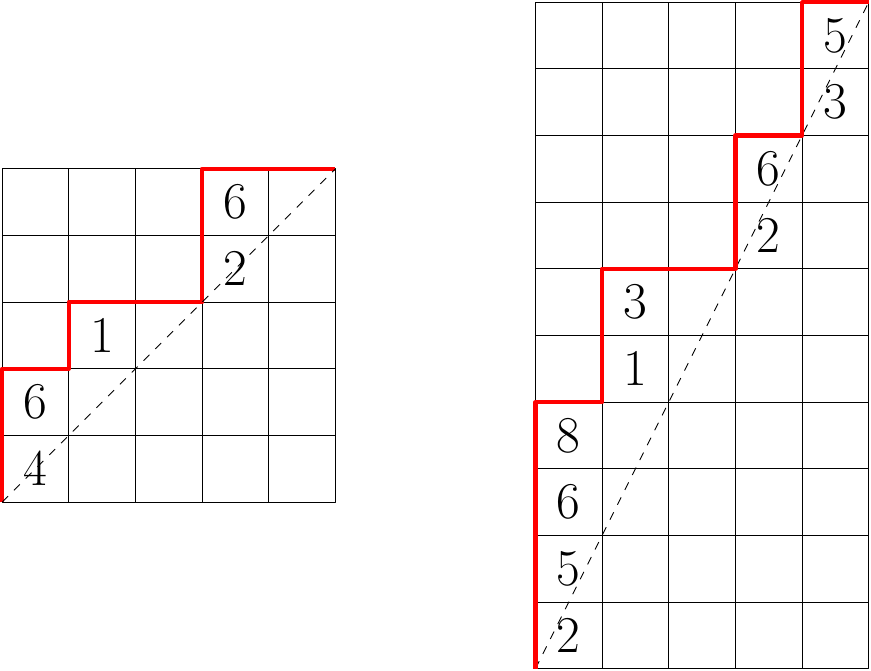}
    \caption{Examples of $(3,3)$ and $(6,3)$ parking functions}
    \label{fig:PF-example}
\end{figure}

\subsection{$\Delta$-Springer fibers}
In a previous work, the first and third authors \cite{GG} studied a family of varieties called $\Delta$-Springer fibers (see also \cite{Griffin,Griffin2,GLW}). We use an equivalent definition following \cite[Lemma 3.3]{GG}.

\begin{defn}
Fix a partition $\mu$ with $|\mu|=k$ and an integer $s\geq \ell(\mu)$, and define $\nu=(n-k)^{s}+\mu$. Also, fix a nilpotent operator $\overline{\gamma}$ on $\bC^{|\nu|}$ with Jordan type $\nu$.
Then the $\Delta$-Springer fiber $Z_{n,\mu,s}$ is defined as the space of partial flags
$$
Z_{n,\mu,s}=\{\bC^{|\nu|} = F_0\supset F_{|\nu|-n}\supset F_{1+|\nu|-n}\supset\cdots\supset F_{|\nu|}=0:\ \dim F_i=|\nu|-i\}
$$
such that $\overline{\gamma}F_i\subset F_i$ and $\JT(\overline{\gamma}|_{F_0/F_{|\nu|-n}}) \leq (n-k)^{s-1}$. 
Note that here we have indexed the parts of the flag in order to match our convention for indexing flags of lattices. Further note that the Jordan type condition may be replaced with $\overline{\gamma}^{n-k}F_0= \mathrm{im}(\overline{\gamma}^{n-k})\subseteq F_{|\nu|-n}$.
\end{defn}

One of the main results of \cite{GG} gives a representation-theoretic description for the cohomology of $\Delta$-Springer fiber.

\begin{thm}[\cite{GG}]
\label{thm: GG}
There is an action of $S_n$ in the cohomology of $Z_{n,\mu,s}$. The corresponding graded Frobenius character equals 
$$
q^{\binom{s-1}{2}(n-k)}\mathrm{Frob}(H^*(Z_{n,\mu,s});q)=
s_{(n-k)^{s-1}}^{\perp}\widetilde{H}_{\nu}(x;q)
$$
where $\widetilde{H}_{\nu}(x;q)$  is the modified Hall-Littlewood polynomial.
\end{thm}

We can relate the variety $Y_{n,k}$ to the $\Delta$-Springer fibers as follows: consider the projection to the affine Grassmannian 
$$
\pi: Y_{n,k,N}\to \aGr
$$
which sends a flag of lattices $\Lambda_{\bullet}$ to $\Lambda_0$. 

\begin{thm}
\label{thm: Intro Delta Springer}
Each fiber of the projection $\pi$ is either empty or isomorphic to a $\Delta$-Springer fiber $Z_{n,\mu,k}$ for some $\mu$.
\end{thm}



To prove Theorem \ref{thm: Intro Delta Springer}, we consider the $K$-dimensional space $\Lambda_0/\epsilon \Lambda_0$, and the partial flag $F_i=\Lambda_i/\epsilon\Lambda_0$. The operator $\overline{\gamma}$ is given by the restriction of $\gamma_{n,k,N}$ to $\Lambda_0/\epsilon \Lambda_0$, and the key Lemma \ref{lem:eval-in-sub} shows that it has Jordan type $\nu=(n-k)^{k}+\mu$ for some $\mu$. Now the conditions defining $Y_{n,k,N}$ translate to the conditions for $Z_{n,\nu,k}$, see Section \ref{sec: geometry} for details.

The proof of Theorem \ref{thm: GG} heavily used the work of Borho and MacPherson \cite{BM} on partial resolutions of the nilpotent cone. Note that by \cite{HottaSpringer} the modified Hall-Littlewood polynomial in the Theorem \ref{thm: GG} can be interpreted as the graded Frobenius character of the homology of the classical Springer fiber, which coincides with the fiber of the projection $X_{n,k,N}\to \aGr$. 

We use this idea to develop a sheaf-theoretic generalization of Theorem \ref{thm: GG}, and use it to prove Theorem \ref{thm:IntroSkewing}(a). 



\begin{rmk}
    Note that the discussion above and Theorem~\ref{thm: Intro Delta Springer} do not imply that $\omega (E_{K,k}\cdot 1)$ and $\omega\Delta'_{e_{k-1}}e_n$ expand positively in terms of Hall-Littlewood symmetric functions, even in the case of $\nabla e_5$. This is because the projection maps from $X_{n,k}$ and $Y_{n,k}$ to $\widetilde{\Gr}$ are not fiber bundles over the cells in the Schubert cell decomposition of $\widetilde{\Gr}$. We invite the reader to check that the fibers $\pi^{-1}(\Lambda)$ of $\pi: X_{5,5}\to\widetilde{\Gr}$ over $\Lambda\in C_{[1,7,8,9,15]}$ depend on $\Lambda$: Generically, $\pi^{-1}(\Lambda)$ is isomorphic to the Springer fiber associated to a nilpotent matrix of Jordan type $(3,1,1)$, but over the torus fixed point of the Schubert cell $\pi^{-1}(\Lambda)$ is isomorphic to the Springer fiber for Jordan type $(2,2,1)$.
\end{rmk}

\subsection{Organization of the paper}

The paper is organized as follows. In Section \ref{sec: background}, we give a combinatorial background on symmetric functions, parking functions and affine permutations. In Section \ref{sec: gamma restricted}, we translate parking function combinatorics and diagonal inversions in terms of a class of affine permutations called $\gamma$-restricted, and we prove several useful lemmas concerning them. 
Sections \ref{sec: geometry}, \ref{sec: Springer action}, and \ref{sec: X dimensions} are focused on the geometry of affine Springer fibers. In Section \ref{sec: geometry} we introduce the varieties $X_{n,k,N}$ and $Y_{n,k,N}$ and prove Theorem \ref{thm:GradedFrob}(a). In Section~\ref{sec: Springer action}, we use Springer theory and work of Borho and MacPherson to show there are compatible symmetric group actions on the Borel-Moore homologies of $X_{n,k,N}$ and $Y_{n,k,N}$. We then prove Theorem~\ref{thm:GradedFrob}(a), a geometric version of the skewing formula. In Section  \ref{sec: X dimensions}, we construct an affine paving of $X_{n,k}$. This allows us to prove Theorem \ref{thm: intro Y} and complete the proof of Theorem \ref{thm:GradedFrob}(b).

\subsection{Future directions}
We conclude the introduction with a few possible further directions. 

In particular, our varieties depend on a parameter $N$. However, we show in Theorems \ref{thm: intro Y} and ~\ref{thm:GradedFrob} that when $N\geq k$ these varieties are independent of $N$ and the graded Frobenius character of their Borel-Moore homologies give the symmetric functions in the Rational Shuffle Theorem and Delta Theorem. However, when $0 \leq N <k$, there exist examples of $X_{n,k,N}$ (respectively $Y_{n,k,N}$) whose Borel-Moore homology groups differ from the terms in the Rational Shuffle Theorem (respectively Delta Theorem). Therefore, one could define new symmetric functions $f_{n,k,N}$ and $g_{n,k,N}$ by
\begin{align*}
f_{n,k,N} &= \rev_q\omega\,\Frob(H_*^{\BM}(X_{n,k,N};q,t)),\\
g_{n,k,N} &= \rev_q\omega\,\Frob(H_*^{\BM}(Y_{n,k,N};q,t)).
\end{align*}
\begin{problem}
    Find combinatorial and Macdonald operator formulas for  $f_{n,k,N}$ and $g_{n,k,N}$ in the cases $0\leq N < k$ that generalize those on either side of the Rational Shuffle and Delta Theorems.
\end{problem}

Another natural question is whether our varieties can be generalized to give geometric interpretations of other symmetric functions coming from Macdonald-theoretic or Elliptic Hall Algebra operators.

\section*{Acknowledgments}

We thank Fran\c{c}ois Bergeron, Eric Carlsson, Mark Haiman, Jim Haglund, Oscar Kivinen, Jake Levinson, Misha Mazin, Anton Mellit, Andrei Negu\cb{t}, Anna Pun, George Seelinger, and Andy Wilson for useful discussions.

\section{Notation and Background}
\label{sec: background}

\subsection{Symmetric functions}
\label{sec: symmetricfunctions}

We refer to \cite{Macdonald} for details on many of the standard definitions in this section.  We will work in the ring of symmetric functions $\Lambda$ in infinitely many variables $x_1,x_2,\ldots$ over $\mathbb{Q}(q,t)$. We will use elementary symmetric functions
$$
e_m=\sum_{i_1<\cdots<i_m}x_{i_1}\cdots x_{i_m},\qquad\qquad e_{\lambda}=\prod_i e_{\lambda_i}
$$
where $\lambda$ is a partition.  We also have monomial symmetric functions 
$$
m_{\lambda}= \sum_{(i_1,\ldots,i_\ell)} x_{i_1}^{\lambda_1}x_{i_2}^{\lambda_2}\cdots x_{i_\ell}^{\lambda_{\ell}}
$$
where $(i_1,\ldots,i_\ell)$ is any tuple of distinct positive integers such that $i_j< i_{j+1}$ whenever $\lambda_j=\lambda_{j+1}$.  Both $\{m_\lambda\}$ and $\{e_\lambda\}$ form a basis of $\Lambda$.

We also use the basis of Schur functions $s_{\lambda}$, defined as
$s_\lambda=\sum K_{\lambda\mu} m_\mu$ where the coefficients $K_{\lambda\mu}$ are the \textit{Kostka numbers}, which count the number of column-strict \textit{Young tableaux} of shape $\lambda$ and content $\mu$.  We draw our Young tableaux in French notation:
$$\young(4,25,1155)$$
and the above tableau has content $(2,1,0,1,3)$, with the $i$th entry indicating the multiplicity of $i$ in the tableau.   Its shape is $(4,2,1)$, indicating the length of each row from bottom to top.

\begin{defn}
The Hall inner product on $\Lambda$ is defined by 
$$
\langle s_{\lambda},s_{\mu}\rangle=\delta_{\lambda,\mu}
$$ and extending by linearity, using the fact that $\{s_\lambda\}$ forms a basis of $\Lambda$.
\end{defn}

The Schur functions and Hall inner product are directly tied to the representation theory of the symmetric group $S_n$.  The irreducible representations $V_\lambda$ of $S_n$ are indexed by partitions $\lambda$ of $n$, and the \textbf{Frobenius map} $\Frob$ takes a representation $V=\bigoplus_\lambda (V_\lambda)^{\oplus c_\lambda}$ to $\sum c_\lambda s_\lambda$.  Clearly $\Frob$ is additive across direct sum, and it has the remarkable property of being multiplicative across induced tensor product:
$$\Frob(\mathrm{Ind}_{S_n\times S_m}^{S_{n+m}} V\otimes W)=\Frob(V)\Frob(W)$$

We make use of a (doubly) graded version of the Frobenius map as follows.
\begin{defn}
    Given a graded $S_n$ module $R=\bigoplus_d R_d$, its \textbf{graded Frobenius character} is $$\grFrob_q(R)=\sum_d \Frob(R_d)q^d.$$ For a doubly-graded module $S=\bigoplus S_{i,j}$, we write $$\grFrob_{q,t}(S)=\sum_{i,j} \Frob(S_{i,j})q^it^j.$$
\end{defn}

We use the adjoint operators to multiplication operators with respect to the Hall inner product. 

\begin{defn}
The operator $s_{\lambda}^{\perp}:\Lambda\to \Lambda$ is defined such that the identity
$$
\langle s_{\lambda}^{\perp}f,g\rangle=\langle f,s_{\lambda}g\rangle$$
holds for all symmetric functions $f,g\in \Lambda$ (resp. $f,g\in \Lambda_k$).
\end{defn}

We will make use of the following lemma from \cite{GG}.

\begin{lem}\cite[Lemma 2.1]{GG}\label{lem:Skewing}
    Given $W$ an $S_n$-module, $S_{n-m}\times S_{m}$ a Young subgroup, and a partition $\mu\vdash m$, then
    \[
    s_\mu^\perp\Frob(W) = \frac{1}{\dim(V_\mu)}\Frob(W^{V_\mu})
    \]
    where $W^{V_\mu}$ is the $V_\mu$-isotypic component of the restriction of $W$ to an $S_{m}$-module, whose Frobenius character is taken as an $S_{n-m}$-module.
\end{lem}

\subsection{Rational parking functions}

Throughout this subsection, let $k$ and $K$ be positive integers such that $k| K$. Let $\RD_{K,k}$ be the set of $K\times k$ rational Dyck paths (height $K$ and width $k$), and let $\PF_{K,k}$ be the labeled parking functions on elements of $\RD_{K,k}$. That is, the set of labelings of the vertical runs of elements of $\RD_{K,k}$ by positive integers such that the labeling weakly increases up each vertical run. 

\begin{defn}
 The \textbf{area} of an element of $\PF_{K,k}$ is the number of whole boxes lying between the path and the diagonal, so that the diagonal does not pass through the interior of the box.
\end{defn}

\begin{defn}
\label{def: dinv}
The \textbf{dinv} statistic (for ``diagonal inversions'') on $\PF_{K,k}$ is defined as $$\dinv(P)=\pathdinv(D)+\tdinv(P)-\maxtdinv(D)$$ where $D$ is the Dyck path of $P$. We define each of these three quantities separately below.
 \end{defn}
 To define $\pathdinv(D)$, recall that the \textit{arm} of a box above a Dyck path in the $K\times k$ grid is the number of boxes to its right that still lie above the Dyck path.  The \textit{leg} is the number of boxes below it that still lie above the Dyck path.

\begin{defn}\label{def:arm-leg}
    The \textbf{pathdinv} of Dyck path $D$ from $(0,0)$ to $(k,K)$ (with $k|K$) is the number of boxes $b$ above the Dyck path of $P$ for which $$\frac{\arm(b)}{\leg(b)+1}\leq k/K< \frac{\arm(b)+1}{\leg(b)}$$
\end{defn}

\begin{rmk}
    We call this statistic $\pathdinv$ here to emphasize that it only depends on the rational Dyck path, and to distinguish it from $\dinv$ of the parking function.  It was simply called $\dinv$ in \cite{RationalShuffle}.
\end{rmk}

To define the $\tdinv$ statistic on rational parking functions, we use the same conditions on pairs of boxes identified in Lemma \ref{lem:box-pairs}, which we call attacking pairs.

\begin{defn}
    A \textbf{diagonal} in the $K$ by $k$ grid, where $k|K$, is a set of boxes whose lower left hand corners pairwise differ by a multiple of $(1,K/k)$.
    
    The \textbf{main diagonal} of the $K$ by $k$ rectangle is the line between $(0,0)$ and $(k,K)$.  Note that we use Cartesian coordinates in the first quadrant for the boxes and label each box's position by the coordinates of its lower left hand corner.
    
    A pair of boxes $a,b$ in a $K$ by $k$ grid (with $k|K$) is an \textbf{attacking pair} if and only if either:
    \begin{itemize}
        \item $a$ and $b$ are on the same diagonal, with $a$ to the left of $b$, or 
        \item $a$ is one diagonal below $b$, and to the right of $b$.
    \end{itemize}
\end{defn}

\begin{rmk}
    When boxes are known to be labeled, we often use the name of the box and its label interchangeably, as in the definition below. 
\end{rmk}

\begin{defn}\label{def:tdinv}
The \textbf{tdinv} of $P\in \PF_{K,k}$ is the number of attacking pairs of labeled boxes $a,b$ in $P$ such that $a<b$.  
\end{defn}

\begin{defn}\label{def:maxtdinv}
    The \textbf{maxtdinv} of a Dyck path $D$ is the largest possible $\tdinv$ of any parking function of shape $D$.  
\end{defn}

Note that $\mathrm{maxtdinv}(D)$ can be achieved by taking $\mathrm{tdinv}(D)$ of the labeling of $D$ obtained by labeling using $1,2,3,\ldots$ across diagonals from left to right, starting from the bottom-most diagonal and moving upwards. Alternatively, $\mathrm{maxtdinv}(D)$ is the number of attacking pairs $(a,b)$ such that $a$ and $b$ are boxes directly to the right of an up-step of $D$.

Define the \textbf{rank} of the box in row $j$ and column $i$ of the $K\times k$ rectangle to be
\[
    \rank(i,j) = 1 + (j-1)k + \left\lfloor \frac{j-1}{n-k+1}\right\rfloor - K(i-1).
\]
In other words, we put $1$ in the box in the SW corner, and fill the rest of the ranks by increasing them by $k$ or $k+1$, respectively, as we go North from $(i,j)$ to $(i,j+1)$ if $j$ is not divisible or divisible by $n-k+1$, respectively. As we go East from $(i,j)$ to $(i+1,j)$, we decrease by $K$.
%
%
The ranks of the boxes in the case of $K=12$ and $k=4$ is shown in Figure \ref{fig: ranks}.

\begin{figure}
\begin{tikzpicture}[scale=0.7]
\draw (0,0)--(0,12)--(4,12)--(4,0)--(0,0);
\draw (1,0)--(1,12);
\draw (2,0)--(2,12);
\draw (3,0)--(3,12);
\draw (0,1)--(4,1);
\draw (0,2)--(4,2);
\draw (0,3)--(4,3);
\draw (0,4)--(4,4);
\draw (0,5)--(4,5);
\draw (0,6)--(4,6);
\draw (0,7)--(4,7);
\draw (0,8)--(4,8);
\draw (0,9)--(4,9);
\draw (0,10)--(4,10);
\draw (0,11)--(4,11);
\draw [dotted] (0,0)--(4,12);
\draw [line width=2] (-1,3)--(4,3);
\draw [line width=2] (-1,6)--(4,6);
\draw [line width=2] (-1,9)--(4,9);
\draw (-0.5,0.5) node {1};
\draw (-0.5,1.5) node {5};
\draw (-0.5,2.5) node {9};
\draw (-0.5,3.5) node {14};
\draw (-0.5,4.5) node {18};
\draw (-0.5,5.5) node {22};
\draw (-0.5,6.5) node {27};
\draw (-0.5,7.5) node {31};
\draw (-0.5,8.5) node {35};
\draw (-0.5,9.5) node {40};
\draw (-0.5,10.5) node {44};
\draw (-0.5,11.5) node {48};
\draw (0.5,3.5) node {2};
\draw (0.5,4.5) node {6};
\draw (0.5,5.5) node {10};
\draw (0.5,6.5) node {15};
\draw (0.5,7.5) node {19};
\draw (0.5,8.5) node {23};
\draw (0.5,9.5) node {28};
\draw (0.5,10.5) node {32};
\draw (0.5,11.5) node {36};
\draw (1.5,6.5) node {3};
\draw (1.5,7.5) node {7};
\draw (1.5,8.5) node {11};
\draw (1.5,9.5) node {16};
\draw (1.5,10.5) node {20};
\draw (1.5,11.5) node {24};
\draw (2.5,9.5) node {4};
\draw (2.5,10.5) node {8};
\draw (2.5,11.5) node {12};
\draw [line width=2,red] (0,0)--(0,6)--(1,6)--(1,7)--(2,7)--(2,11)--(3,11)--(3,12)--(4,12);
\end{tikzpicture}
\caption{Ranks for $(K,k)=(12,4)$}
\label{fig: ranks}
\end{figure}

The following is easily verified and we omit the proof.

\begin{prop}
\label{prop: ranks}
All ranks above the diagonal are strictly positive, and pairwise distinct. The ranks in different rows have different remainders mod $K$. All ranks below or on the diagonal are nonpositive. 
\end{prop}

We will need the following statement relating ranks to attacking pairs.

\begin{lem}\label{lem:box-pairs}
Suppose $a,b$ are two boxes inside the rectangle. Then $\rank(a)<\rank(b)\le \rank(a)+k$ if and only if $(a,b)$ form an attacking pair.
\end{lem}

\begin{proof}
First observe that each diagonal has at most one box in each block and that the ranks on the same diagonal increase by $1$ from left to right since  for each box $(i,j)$,
\[
    \rank(i+1,j+K/k) = \rank(i,j) + k(K/k)+1 - K = \rank(i,j)+1.
\]

For the reverse implication, we consider the following cases:

\textbf{Case 1.} Suppose the boxes containing $a$ and $b$ form an attacking pair on the same diagonal, with $a$ to the left of $b$.  Then we have $a<b<a+k$ because the ranks on a diagonal increase by $1$ and there are $k$ columns.

\textbf{Case 2.} Suppose the boxes containing $a$ and $b$ form an attacking pair with $a$ to the right of $b$ where $a$ is in one lower diagonal than $b$.  Letting $x$ be the entry below $b$, we have $a=x+j$ for some $j$ with $1\le j<k$ since $x$ and $a$ are on the same diagonal.  Then $b$ is either equal to $x+k$ or $x+k+1$, so $a<b\le a+k$.

We now show that if two squares do not form an attacking pair, then one of the inequalities is not satisfied.  If $a$ is to the right of $b$ in its diagonal, or left of $b$ in one higher diagonal, the inequality $a<b$ is not satisfied by the cases above.  So we can assume the two boxes are at least two diagonals apart.  But if the largest entry in the lower diagonal is $e$, then the smallest entry two diagonals up is equal to $e+1+k>e+k$ and so we cannot have the inequality $b\le a+k$ satisfied.
\end{proof}






\begin{defn}\label{def:delta-num}
    We write $\delta_{K,k}$ for the number of boxes that lie fully above the diagonal in the $K\times k$ rectangle, which is equal to $(k\cdot K-K)/2=(k-1)K/2$.
\end{defn}

We now show a correspondence between certain attacking pairs and the complement squares to those that contribute to pathdinv.  To do so, we first rephrase the definition of pathdinv in a way that we used in \cite{GGG1}.  The proof simply follows from analyzing both sides of the inequality in Definition \ref{def:arm-leg}, and we omit it.

\begin{lem}\label{lem:alt-pathdinv}
    The statistic $\pathdinv(D)$ is the number of boxes $b$ above $D$ such that, if $s=K/k$ is the slope of the diagonal, we have $$\leg(b)\in \{s\cdot \arm(b)-1,s\cdot\arm(b),s\cdot\arm(b)+1,s\cdot\arm(b)+2,\ldots,s\cdot\arm(b)+(s-1)\}.$$
\end{lem}

\begin{figure}
    \centering
    \includegraphics{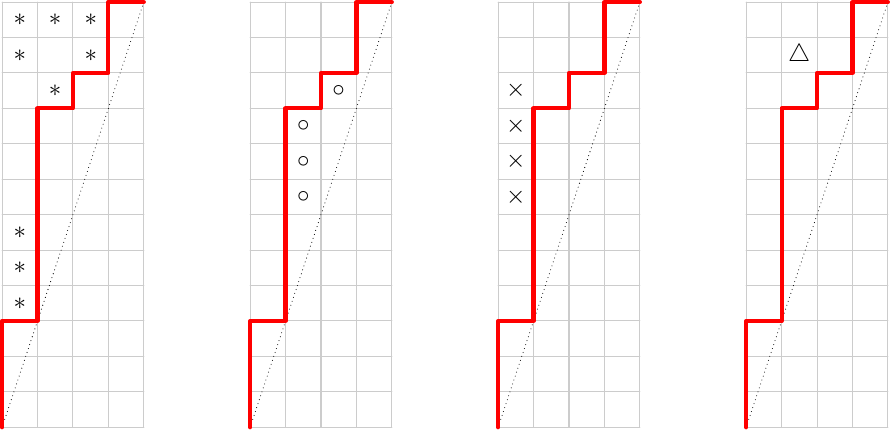}
    \caption{At left, computing pathdinv of the red path $D$ as the number of boxes marked with $\ast$, whose arm and leg satisfy the conditions of Lemma \ref{lem:alt-pathdinv}.  We therefore have $\pathdinv(D)=9$. The remaining figures show the three types of complementary boxes as used in the proof of Lemma \ref{lem: codinv of a  path}.}
    \label{fig:example-pathdinv}
\end{figure}

\begin{lem}
\label{lem: codinv of a  path}
For a Dyck path $D$ the difference
$\delta_{K,k}-\pathdinv(D)$ equals the number of attacking pairs $a,b$ of boxes such that $a$ is to the left of a vertical step in $D$, and $b$ is between $D$ and the main diagonal.
\end{lem}

\begin{proof}
    We outline the proof of essentially the same fact shown in \cite[Lemma 2.22]{GorskyMazin}, though we note that their figures are flipped upside down and transposed from ours, and we translate the main steps here in our notation. 
    
    In Figure \ref{fig:example-pathdinv}, the boxes contributing to $\pathdinv$ are marked with $\ast$. We know that $\delta_{K,k}-\pathdinv(D)$ is the number of boxes fully above the diagonal that are not marked with $\ast$, and we sort them into three types:
    \begin{itemize}
        \item \textbf{Below $D$:} We mark these boxes with a $\circ$.
        \item \textbf{Above $D$, ``too high'':} These squares $b$ have $\leg(b)\ge s\cdot \arm(b)+s$, in other words, their leg is too long to be marked with $\ast$. We mark these with $\times$.
        \item \textbf{Above $D$, ``too low'':} These squares $b$ have $\leg(b)\le s\cdot \arm(b)-2$, in other words, their leg is too short.  We mark these with $\triangle$.
    \end{itemize}

    For each $\triangle$ entry, we draw a line to the right until we cross a vertical step of $D$; let $a$ be the box just to the left of this vertical step.  Then by the height condition for $\triangle$, there is a box $b$ above the diagonal and below $D$ in the same column as $\triangle$ that is one diagonal above $a$.  We match the $\triangle$ square to this attacking pair $(a,b)$, as shown at left below.  Conversely, every such off-diagonal pair $(a,b)$ corresponds to a triangle.

\begin{center}
    \includegraphics{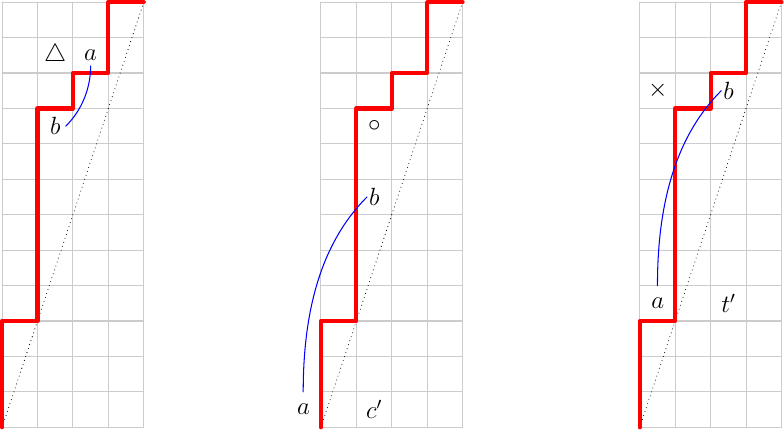}
\end{center}

    We now consider a square $c$ marked with $\circ$.  We first reflect $c$ vertically as follows: let $d$ be the number of squares vertically between $c$ and the path $D$, and define $c'$ to be the square in the same column as $c$ that has exactly $d$ squares vertically below it in the grid (starting at the bottom below the diagonal).  Then, let $a$ be the square to the left of the vertical step of $D$ in the row of $c'$, and $b$ the square in the column of $c'$ on the same diagonal as $a$.  By an analysis of the heights involved, one sees that $b$ is indeed a square between the path and the diagonal, and so $(a,b)$ form an attacking pair of the desired form.  We match $c$ with the attacking pair $(a,b)$ as shown at middle above.  By the work in \cite{GorskyMazin}, this association is bijective with the attacking pairs formed by taking a square $b$ between the diagonal and $D$ and matching it with the entry $a$ on its diagonal just outside of $D$ that is farthest to the left.

    The remaining pairs are matched with $\times$'s as shown at right above, via the following algorithm.  Given a square $t$ marked with $\times$, we move exactly $\left \lfloor\frac{\leg(t)}{s}\right\rfloor$ columns to the right, and in the new column we find the square $t'$ below $D$ whose vertical distance to $D$ equals $\leg(t)$.  Then the row of $t'$ determines the entry $a$ to its left, and the entry $b$ is the unique entry in the column of $t'$ in the same diagonal as $a$.  This again is a bijection with the remaining pairs by the work in \cite{GorskyMazin}. 
\end{proof}

\subsection{Stacked parking functions}
\label{sec: stacks}

We use the notation of \cite{GGG1,HRW} here, and recall the definition of a \textit{stacked} parking function that can be used to reformulate the Rise version of the Delta conjecture.  

\begin{defn}
A \textbf{stack} $S$ of boxes in an $n\times k$ grid is a subset of the grid boxes such that there is one element of $S$ in each row, at least one in each column, and each box in $S$ is weakly to the right of the one below it.

A \textbf{stacked parking function} with respect to $S$ is a labeled lattice path $P$ with north and east steps from $(0,0)$ to $(k,n)$ such that each box of $S$ lies below the path of $P$, and the labeling is strictly increasing up each column.  (See the right hand diagram in Figure \ref{fig:map-F}.)
\end{defn}

We write $\LD(S)$ is the set of stacked parking functions with respect to $S$, and
\[
\LD_{n,k}^{\stack} \coloneqq \bigcup_{S\in \Stack_{n,k}} \LD(S).
\]


We will also need the following construction from \cite{GGG1} relating rational parking functions with the stacked ones. Set $K = k(n-k+1)$.  Given a standard parking function in the $K\times k$ rectangle, we call a label $a$ {\bf big} if $a>n$, and {\bf small} if $a\le n$.  We denote by $b_i$ the number of big labels in the $i$-th column, and by $s_i$ the number of small labels. Note that 
\begin{equation}
\label{eq: sum small big}
\sum_{i=1}^{k} s_i=n,\ \sum_{i=1}^{k} b_i=K-n
\end{equation}

\begin{defn}
\label{def: admissible for b}
A $(K,k)$-parking function is called {\bf admissible} if $b_i\le n-k$ for all $i$.
\end{defn}

We construct a map $F$ from the set of standard admissible $(K,k)$ parking functions to the set $\LD_{n,k}^{\stack}$ as follows:

\begin{figure}
    \centering
    \includegraphics{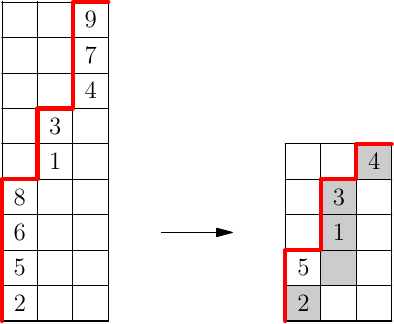}
    \caption{The map $F$ that shrinks a $(K,k)$-parking function to a stacked parking function by removing the big labels.  Here $k=3$ and $n=6$, so $K=9$.}
    \label{fig:map-F}
\end{figure}

\begin{itemize}
\item The parking function $F(\pi)$ is obtained by erasing all big labels in $\pi$ and the north steps to the left of them.
\item In particular, the length of the $i$th vertical run of the lattice path for $F(\pi)$ is equal to $s_i$.
\item The height of the stack in column $i$ is given by $h_i=n-k+1-b_i$.
\end{itemize}

See Figure \ref{fig:map-F} for an example.

\begin{lem}\cite[Lemma 4.7]{GGG1}
\label{lem: admissible to stacks}
The stacked parking function $F(\pi)$ is well defined in $\LD_{n,k}^{\stack}$. We have a bijection between $\LD_{n,k}^{\stack}$ and the quotient of the set of admissible $(K,k)$ parking functions by the action of $S_{K-n}$ which permutes the big labels.
\end{lem}

One can define the analogues of $\area$ and $\dinv$ statistics for stacked parking functions. These appear in the Delta Theorem, but we do not need them in this paper and refer to \cite[Section 4.1]{GGG1} for all definitions and a detailed discussion.

\subsection{Affine permutations}  We now establish notation on affine permutations.

\begin{defn}
    An (extended) \textbf{affine permutation} in $\widetilde{S_K}$ is a bijection $\omega:\mathbb{Z}\to \mathbb{Z}$ such that, for all $i\in \mathbb{Z}$, $$\omega(i+K)=\omega(i)+K.$$  
We will often write the affine permutations in {\bf window notation}
$$
\omega=\left[\omega(1),\ldots,\omega(K)\right].
$$
\end{defn} 

In order for $\omega$ to be well defined, we need $\omega(1),\dots,\omega(K)$ to have pairwise distinct remainders modulo $K$. By reducing $\omega(i)$ modulo $K$, we get a projection $\widetilde{S_K}\to S_K$.

\begin{defn}
    The \textbf{degree} of an affine permutation $\omega$ is
\begin{equation}
\label{def: degree omega}
\deg \omega=\frac{1}{K}\sum_{i=1}^K (\omega(i)-i)=\frac{1}{K}\left(\sum_{i=1}^{K}\omega(i)-\frac{K(K+1)}{2}\right).
\end{equation}
\end{defn}
One can check that $\deg\omega$ is always an integer and 
$
\deg(\omega_1\omega_2)=\deg \omega_1+\deg \omega_2.
$
\begin{defn}
Given $\lambda\in \bZ^K$, we define the translation by 
$$
\ttt_\lambda:=\left[1+K\lambda_1,\ldots,K+K\lambda_K\right].
$$
\end{defn}
We have
$
\deg(\ttt_\lambda)=|\lambda|=\sum_{i=1}^{K}\lambda_i.
$
Furthermore, any affine permutation can be uniquely written as
\begin{equation}
\label{eq: omega decomposition}
\omega=\ttt_{\lambda}\ww=[\ww(1)+K\lambda_{\ww(1)},\ldots,\ww(K)+K\lambda_{\ww(K)}],\ \lambda\in \bZ^K,\ww\in S_K. 
\end{equation}

\begin{defn}
\label{def: positive permutation}
We say that an affine permutation $\omega=[\omega(1),\ldots,\omega(K)]$ is {\bf positive} if $\omega(i)>0$ for $i=1,\ldots,K$. Equivalently,
$\omega=\ttt_{\lambda}\ww$ and $\lambda_i\ge 0$ for $i=1,\ldots,K$. We denote the set of positive affine permutations by $\widetilde{S}_K^{+}$.
\end{defn}

\begin{defn}
\label{def: normalized permutation}
We say that an affine permutation $\omega$ is {\bf normalized}, if $1\le \omega^{-1}(1)\le K$, that is, the values $\omega(1),\ldots,\omega(K)$ in the window contain $1$. Equivalently, we require that $\lambda_1=0$.
We denote by $\widetilde{S}_K^{0}$ the set of normalized affine permutations, and by 
$\widetilde{S}_K^{+,0}$ the set of positive and normalized affine permutations.
\end{defn}

\begin{rmk}
\label{rem: right action preserves positive}
We have a {\bf right} action of $S_K$ on $\widetilde{S_K}$. Note that for ${\bf u}\in S_K$ we have
$$
(\ttt_{\lambda}\ww){\bf u}^{-1}=[\ww {\bf u}^{-1}(1)+K\lambda_{\ww {\bf u}^{-1}(1)},\ldots,\ww {\bf u}^{-1}(K)+K\lambda_{\ww {\bf u}^{-1}(K)}].
$$
In particular, the right action of $S_K$ preserves the sets $\widetilde{S}_K^+,\widetilde{S}_K^{0}$ and $\widetilde{S}_K^{+,0}$. 
\end{rmk}

\section{Combinatorics of $\gamma$-restricted permutations}\label{sec: gamma restricted}

In this section, we establish some combinatorial results on a special affine permutation $\gamma$ that will correspond to the nil-elliptic element used to construct the affine Springer fibers in the sections below.  The facts proven here will help with dimension counting.

\subsection{Definition and setup} 
We will need a special affine permutation $\gamma=\gamma_{n,k,N}\in \widetilde{S}_K$ defined as follows:
\begin{equation}
\label{eq: def gamma}
\gamma(x) = 
    \begin{cases}
    x+k & \text{if }1\leq x\leq (n-k)k\\
    x+k+1 & \text{if } (n-k)k < x < K\\
    1+K(N+1) & \text{if }x=K.
    \end{cases}
\end{equation}

\begin{rmk}\label{rmk:RankAbove}
    Observe that for a box in row $j$ and column $i$ in the $K\times k$ rectangle such that $j<K$, the rank of the box $(i,j+1)$ immediately above $(i,j)$ satisfies
    \[
    \rank(i,j+1) = \gamma(\rank(i,j)).
    \]
    Furthermore, the boxes with $j<K$ which are above the main diagonal are exactly those for which $x=\rank(i,j)$ is not divisible by $K$, so that either $\gamma(x) = x+k+1$ (if $j$ is divisible by $n-k+1$) or $\gamma(x) = x+k$ (if $j$ is not divisible by $n-k+1$).
\end{rmk}

\begin{lem}
\label{lem: gamma is a permutation}
The affine permutation $\gamma$ is well defined in $\widetilde{S_K}$. Its projection to $S_K$ is a single $K$-cycle.
\end{lem}

\begin{proof}
It is sufficient to prove the second claim. By iterating $\gamma$ and reducing modulo $K$, we get
$$
1\mapsto k+1\mapsto\cdots\mapsto (n-k)k+1\mapsto K+2\cong 2\mapsto k+2\mapsto \cdots 
$$
$$
\mapsto k\mapsto 2k \mapsto\cdots\mapsto (n-k)k\mapsto K\mapsto  1+K(N+1)\cong 1\mod K.
$$ This completes the proof.
\end{proof}

\begin{example}\label{ex: gamma}
    For $k=4$, $K=12$, $N=3$, we have $$\gamma=[5,6,7,8,9,10,11,12,14,15,16,49].$$  Its projection to $S_K$ is $5,6,7,8,9,10,11,12,2,3,4,1$ in list notation, which in cycle notation is $$(1\,\,5\,\,9\,\,2\,\,6\,\,10\,\,3\,\,7\,\,11\,\,4\,\,8\,\,12).$$
\end{example}

\begin{defn}
\label{def: gamma restricted}
We say that an affine permutation $\omega$ is {\bf $\gamma$-restricted} if:
\begin{itemize}
\item $\omega$ is  positive and normalized, and
\item $\omega^{-1}(x)<\omega^{-1}(\gamma(x))$ for all $x\in \bZ$.
\end{itemize}
\end{defn}

\begin{example}
The affine permutation
    $$\omega=[1,5,11,16,6,9,20,10,12,14,15,19]$$ is $\gamma$-restricted for $\gamma$ as in Example \ref{ex: gamma}, because $5$ is to the right of $1$, $6$ is to the right of $2$ (which appears to the left of the window), $7$ is to the right of $3$ (which appear consecutively just to the left of the window), and so on.  
\end{example}

\begin{example}\label{ex:omega-inverse}
    We will often consider the inverse of a $\gamma$-restricted affine permutation; for the example above, we have $$\omega^{-1}=[1,-2,-1,-8,2,5,0,-5,6,8,3,9].$$
\end{example}

\begin{defn}
    An \textbf{inversion} of an affine permutation $\omega$ is a pair $(\alpha,\beta)$ of entries such that  $\alpha$ lies in the window (that is, $1\le \omega^{-1}(\alpha)\le K$), $\alpha<\beta$, and $\beta$ occurs to the left of $\alpha$ (that is, $\omega^{-1}(\beta)<\omega^{-1}(\alpha)$).  Note that $\beta$ may be to the left of the window.

    We write $\inv(\omega)$ to denote the number of inversions of $\omega$.
\end{defn}

\begin{example}

The inversions in the affine permutation $\omega$ above are $$(1,2),(1,3),(1,4),(1,7),(1,8),(5,7),(5,8),(6,7),(6,8),(6,11),(6,16),$$
$$(9,11),(9,16),(10,11),(10,16),(10,20),(12,16),(12,20),$$
$$(14,16),(14,20),(15,16),(15,20),(19,20).$$
   Thus $\inv(\omega)=23$.  Note that $\inv(\omega)=\inv(\omega^{-1})$ for any affine permutation $\omega$, and so above we would find $\inv(\omega^{-1})=23$ as well.
\end{example}

\subsection{Relating rational parking functions and affine permutations}\label{sec:P-w}
Let $\pi$ be a parking function with Dyck path $D$. We consider the following additional labeling: to the right of each vertical step of $D$ we write the parking function label as usual, and to the left we write the rank of the corresponding box, as in Figure \ref{fig:parking-func-of-w}.

\begin{figure}
    \centering
\begin{tikzpicture}[scale=0.8]
\draw (0,0)--(0,12)--(4,12)--(4,0)--(0,0);
\draw (1,0)--(1,12);
\draw (2,0)--(2,12);
\draw (3,0)--(3,12);
\draw (0,1)--(4,1);
\draw (0,2)--(4,2);
\draw (0,3)--(4,3);
\draw (0,4)--(4,4);
\draw (0,5)--(4,5);
\draw (0,6)--(4,6);
\draw (0,7)--(4,7);
\draw (0,8)--(4,8);
\draw (0,9)--(4,9);
\draw (0,10)--(4,10);
\draw (0,11)--(4,11);
\draw [dotted] (0,0)--(4,12);
\draw [line width=2] (0,3)--(4,3);
\draw [line width=2] (0,6)--(4,6);
\draw [line width=2] (0,9)--(4,9);
\draw (0.5,0.5) node {1};
\draw (0.5,1.5) node {2};
\draw (0.5,2.5) node {6};
\draw (0.5,3.5) node {10};
\draw (1.5,4.5) node {5};
\draw (1.5,5.5) node {8};
\draw (1.5,6.5) node {11};
\draw (1.5,7.5) node {12};
\draw (2.5,8.5) node {3};
\draw (2.5,9.5) node {4};
\draw (2.5,10.5) node {7};
\draw (3.5,11.5) node {9};
\draw [line width=2,red] (0,0)--(0,4)--(1,4)--(1,8)--(2,8)--(2,11)--(3,11)--(3,12)--(4,12);

\draw(-0.5,0.5) node {$\mathbf{1}$};
\draw(-0.5,1.5) node {$\mathbf{5}$};
\draw(-0.5,2.5) node {$\mathbf{9}$};
\draw(-0.5,3.5) node {$\mathbf{14}$};

\draw(0.5,4.5) node {$\mathbf{6}$};
\draw(0.5,5.5) node {$\mathbf{10}$};
\draw(0.5,6.5) node {$\mathbf{15}$};
\draw(0.5,7.5) node {$\mathbf{19}$};

\draw(1.5,8.5) node {$\mathbf{11}$};
\draw(1.5,9.5) node {$\mathbf{16}$};
\draw(1.5,10.5) node {$\mathbf{20}$};

\draw(2.5,11.5) node {$\mathbf{12}$};
\end{tikzpicture}
    \caption{The labeled rational Dyck path $\pi$ corresponding to $\omega_\pi = [1,5,11,16,6,9,20,10,12,14,15,19]$. The ranks are in bold to the left of the Dyck path, and the parking function is to the right of the path.}
    \label{fig:parking-func-of-w}
\end{figure}

\begin{defn}
To a parking function $\pi$ we associate an affine permutation $\omega_\pi$ such that $\omega_\pi(i)$ is the rank of the box in the same row as, and just to the left of, the parking function label $i$.
\end{defn}

\begin{exa}
The parking function $\pi$ in Figure~\ref{fig:parking-func-of-w} corresponds to the affine permutation 
$$\omega=\omega_\pi=[1,5,11,16,6,9,20,10,12,14,15,19].$$
which is $\gamma$-restricted as shown above. 
 Note that $\omega=\ttt_{\lambda}\ww$ where 
$$\ww=[1,5,11,4,6,9,8,10,12,2,3,7]$$ and 
$\lambda=(0,1,1,1,0,0,1,1,0,0,0,0)$.
\end{exa}

\begin{lem}
\label{lem: pf bijection}
The map $\pi\mapsto \omega_{\pi}$ is a bijection between the set of $(K,k)$ parking functions and the set of $\gamma$-restricted affine permutations, for any fixed $N\ge 0$.
\end{lem}

\begin{proof}
We first show that $\omega=\omega_\pi$ is always $\gamma$-restricted. 
 By Proposition \ref{prop: ranks} the values  $\omega(i)$ are all positive and have pairwise distinct remainders modulo $K$, so 
$\omega$ is a well-defined positive affine permutation. Furthermore, the vertical step at the southwest corner has rank 1, so $\omega$ is normalized.

Next, consider two vertical steps in consecutive rows of $D$ with parking function labels $i$ and $j$, and let $m\ge 0$ be the horizontal distance between them (see Figure \ref{fig: i and j}). Let $x=\omega(i)$. Then  $\gamma(x)$ is the rank of the box just above $x$ by Remark~\ref{rmk:RankAbove} since the parking function label $i$ is not in the top row of the $K\times k$ rectangle (see Remark \ref{rmk:RankAbove}). 

\begin{figure}[ht!]
\begin{tikzpicture}
\draw [line width=2,red] (0,0)--(0,1);
\draw [line width=2,red] (2,1)--(2,2);
\draw (0,1)--(2,1);
\draw (0,1)--(0,2);
\draw (-0.5,0.5) node {$x$};
\draw (-0.5,1.5) node {$\gamma(x)$};
\draw (0.5,0.5) node {$i$};
\draw (2.5,1.5) node {$j$};
\draw [<->] (0,1.2)--(2,1.2);
\draw (1,1.4) node {$m$};
\end{tikzpicture}
\caption{}
\label{fig: i and j}
\end{figure}

We have $\gamma(x)=\omega(j)+mK=\omega(j+mK)$, so that $\omega^{-1}(\gamma(x))=j+mK$.  If $m=0$, then $\omega^{-1}(\gamma(x))=j>i=\omega^{-1}(x)$ by the parking function condition.  If $m>0$, then $j+mK>K\ge i$, so again we have $\omega^{-1}(\gamma(x))>\omega^{-1}(x)$.  Finally, for $x=\omega(i)$ in the top row we have $\gamma(x)=1+(N+1)K$. Since $\omega^{-1}(1)\ge 1$, we get $$\omega^{-1}(\gamma(x))=\omega^{-1}(1+(N+1)K)=\omega^{-1}(1)+(N+1)K\ge 1+K>i=\omega^{-1}(x).$$

We have shown the image $\omega_\pi$ of $\pi$ under the map is a $\gamma$-restricted affine permutation; we now construct the inverse map. Given a $\gamma$-restricted affine permutation $\omega$, we construct the rational Dyck path by placing vertical steps to the right of the ranks corresponding to the values $\omega(i)$ for $1\le i\le K$. Since $\omega$ is positive, by Proposition \ref{prop: ranks} all the vertical steps are to the left of the diagonal. Since $\omega$ is normalized, there is a vertical step at the southwest corner of the rectangle.  Since $\omega$ is an affine permutation, there is one vertical step in each row.

We now show that the vertical steps move weakly to the right from bottom to top, forming a Dyck path by connecting the runs of vertical steps horizontally.  For a pair of consecutive vertical steps labeled as in Figure \ref{fig: i and j} (where $m$ may be negative) the inequality $\omega^{-1}(x)<\omega^{-1}(\gamma(x))$ is equivalent to $j+mK>i$.  Thus $m\ge 0$, so we do indeed have a well-defined Dyck path.  

We then label the box to the right of the vertical step labeled $\omega(i)$ by $i$ for each $i$.  To check that this gives a parking function, either $m=0$ and $j>i$ (and $\pi$ satisfies the parking function condition) or $j$ is to the right of $i$. 
 The reverse map clearly inverts the forward map, so the proof is complete.
\end{proof}

The following lemma is clear from Lemma~\ref{lem: pf bijection}, which is a fact we will need in later sections.

\begin{lem}\label{lem: A Wild Conductor Appears}
    If $\omega$ is $\gamma$-restricted, then $\omega(i) \leq kK$ for all $1\leq i \leq K$.
\end{lem}

We now use Lemma \ref{lem: codinv of a  path} to translate the $\dinv$ statistic into the language of affine permutations as follows.

\begin{lem}\label{lem:deltaminusdinv}  For any $\pi\in \PF_{K,k}$ and $\omega=\omega_{\pi}$, we have
    \[
    \delta_{K,k}-\dinv(\pi) = |\{(\alpha,\beta)\mid 1\leq \omega^{-1}(\alpha)\leq K,\, \omega^{-1}(\beta) < \omega^{-1}(\alpha),\, \alpha<\beta\leq \alpha+k\}|
    \]
\end{lem}

\begin{proof}
Let $D$ be the Dyck path of $\pi$.  Since $\dinv(\pi)=\pathdinv(D)+\tdinv(\pi)-\maxtdinv(D)$, we have
$$\delta_{K,k}-\dinv(\pi)=\left(\delta_{K,k}-\pathdinv(D)\right)+(\maxtdinv(D)-\tdinv(\pi)).
$$

By Lemma \ref{lem: codinv of a  path}, the quantity $(\delta_{K,k}-\pathdinv(D))$ counts the number of attacking pairs $(a,b)$ of boxes such that $a$ is to the left of a vertical step in $D$ and $b$ is between $D$ and the main diagonal.  We claim these pairs correspond to the pairs $(\alpha,\beta)$ satisfying the stated conditions and also having $\omega^{-1}(\beta)\le 0$.

Indeed, let $\alpha$ be the rank of box $a$ and $\beta$ the rank of box $b$; then $\omega^{-1}(\alpha)$ is the parking function label of $\pi$ to the right of $a$.  Since $\beta$ is below $D$, we have that $\beta+qK$ is a rank of a vertical step for some $q\ge 1$, so we have that $\omega^{-1}(\beta+qK)=\omega^{-1}(\beta)+qK$ is a parking function label between $1$ and $K$ inclusive.  But then that means $\omega^{-1}(\beta)\le 0$.  In fact, this analysis shows $\omega^{-1}(\beta)\le 0$ if and only if $\beta$ is below $D$.  By Lemma \ref{lem:box-pairs}, we have $1\le \alpha<\beta\le \alpha+k$, and since we assumed $\beta$ is above the diagonal we count all such pairs.  

We now show the quantity $(\maxtdinv(D)-\tdinv(\pi))$ counts the remaining pairs $(\alpha,\beta)$, that is, with $\omega^{-1}(\beta)\ge 1$.  In this case $\alpha$ and $\beta$ both correspond to ranks of vertical steps of the parking function, with labels $\omega^{-1}(\alpha)$ and $\omega^{-1}(\beta)$, and the inequalities mean that they are precisely the attacking pairs that do not contribute to $\tdinv(\pi)$.  This completes the proof.
\end{proof}

\subsection{Properties of $\gamma$-restricted affine permutations}

Throughout this subsection we assume that $\omega^{-1}$ is $\gamma$-restricted and $N\ge 1$. Because we are making the inverse $\omega^{-1}$ be $\gamma$-restricted, this means for $\omega$ that:
\begin{itemize}
\item $\omega(x)<\omega(\gamma (x))$ for all $x\in \mathbb{Z}$.
\item $1\le \omega(1)\le K$, that is, $\omega^{-1}$ is normalized
\item $\omega^{-1}(i)>0$ for $i=1,\ldots,K$. In other words, if $1\le \omega(\alpha)\le K$ then $\alpha>0$.
\end{itemize}

 We start with several technical lemmas about $\gamma$ and $\omega^{-1}$.

 \begin{lem}
\label{lem: beta mK}
     If $\beta=mK$, then $\gamma(\beta)=\beta+KN+1$ for any $m$.  Moreover, if $\omega^{-1}$ is $\gamma$-restricted, we have $\omega(\gamma(\beta))\ge \gamma(\beta)$.
 \end{lem}
 
\begin{proof}
    We have, since $\gamma\in \widetilde{S}_K$, $$\gamma(\beta)=\gamma(mK)=\gamma(K)+(m-1)K=1+K(N+1)+\beta-K=\beta+KN+1.$$
    We also have 
$$
\omega(\gamma(\beta))=\omega(mK+NK+1)=mK+NK+\omega(1)\ge mK+NK+1=\gamma(\beta)
$$ as desired.
\end{proof}

\begin{cor}
\label{cor: beta mK}
Suppose that $\omega^{-1}$ is $\gamma$-restricted, $\omega(\gamma(\beta))<K$ and $\gamma(\beta)>k$.
Then $\beta$ is not divisible by $K$.
\end{cor}

\begin{proof}
Assume by way of contradiction that $\beta=mK$. By Lemma \ref{lem: beta mK} we have
$$
k<\gamma(\beta)\le \omega(\gamma(\beta))<K.
$$
On the other hand, $\gamma(\beta)=1\pmod K$,  a contradiction.
\end{proof}

\begin{lem}
\label{lem: gamma almost increasing}
We have the following two facts.
\begin{enumerate}
    \item Suppose $\alpha<\beta$ and $\alpha\neq mK$ for all $m\in \mathbb{Z}$. Then $\gamma(\alpha)<\gamma (\beta)$.
    \item Suppose $\gamma (\alpha)<\gamma (\beta)$ and $\beta\neq mK$ for all $m\in \mathbb{Z}$. Then $\alpha<\beta$.
\end{enumerate}
\end{lem}

\begin{proof}
For (1), if $\beta=m'K$ for some $m'$, then since $\gamma\in \widetilde{S}_K$ we have by Lemma \ref{lem: beta mK} that $$\gamma(\beta)=\beta+KN+1\ge \beta+K+1$$ and $\gamma(\alpha)\le \alpha+k+1$, so $\gamma(\alpha)<\gamma(\beta)$. 

Otherwise, $\gamma(\alpha) \in \{\alpha +k,\alpha +k+1\}$ and $\gamma(\beta) \in \{\beta+k,\beta+k+1\}$. Since $\alpha<\beta$ we have $\alpha+1\le \beta$ and $\gamma(\alpha)\le \alpha+k+1\le \beta+k\le \gamma(\beta)$. But $\gamma$ is a bijection, so $\gamma(\alpha)<\gamma(\beta)$.

For (2), if $\alpha = m'K$ for some $m'$, then $\gamma(\alpha)=\alpha+KN+1$ by Lemma \ref{lem: beta mK}. Since $\gamma(\alpha)<\gamma(\beta)$, and $\beta \neq mK$, we get
    \[
    \alpha+K+1 \leq \alpha+NK+1 =\gamma(\alpha) < \gamma(\beta) \leq \beta + k +1
    \]
    and thus $\alpha < \beta + (k-K) \leq \beta$.
    
Otherwise, suppose $\alpha\neq m'K$. Since $\beta \neq mK$ for all $m$, then similarly to part (1) we have 
$$\alpha + k \leq \gamma(\alpha) < \gamma(\beta) \leq \beta+k+1,$$ so $\alpha\leq \beta$, which in turn implies $\alpha<\beta$.    
\end{proof}

For the following lemmas, we define two sets of pairs of integers depending on an affine permutation $\omega$ (whose inverse is $\gamma$-restricted) that correspond to certain inversions, and that will give the dimension count of one of our geometric spaces. Define the sets
\begin{equation}
\label{eq: def A}
    A_{\omega} = \{(\alpha,\beta)\mid 1\leq \omega(\alpha)\leq K,\, \alpha < \beta,\, \omega(\gamma(\beta)) < \omega(\alpha)\}
\end{equation}
and 
\begin{equation}
\label{eq: def B}
B_{\omega}=\{ (\alpha,\beta)\mid 1\leq \omega(\alpha)\leq K,\, \omega(\beta) < \omega(\alpha),\,\alpha+k<\beta\}.
\end{equation}

\begin{example}
Since we have switched notation in this section so that $\omega^{-1}$ is $\gamma$-restricted, in this example we define $\omega$ to be the inverse of our $\gamma$-restricted example above, namely 
$$\omega=[1, -2, -1, -8, 2, 5, 0, -5, 6, 8, 3, 9].$$ We have that the values of $\alpha$ such that $1\le \omega(\alpha)\le K=12$ are $$\alpha=1,5,6,9,10,11,12,14,15,16,19,20,$$ which are precisely the elements in the window of $\omega^{-1}$.  Recall that $$\gamma=[5,6,7,8,9,10,11,12,14,15,16,49]$$ in this running example as well, so an example of a pair $(\alpha,\beta)\in A_\omega$ is $(1,3)$, since $1<3$ and $$\omega(\gamma(3))=\omega(7)=0<1=\omega(1).$$
 We then have, by similar computations, that $$A_{\omega}=\{(1,3),(1,4),(6,7),(6,11),(9,11),(10,11),(10,16),(12,16)\}.$$
For $B_\omega$, we consider the same set of values for $\alpha$ but now we find $\beta$ such that $\beta>\alpha+4$ and $\omega(\beta)<\omega(\alpha)$.  For instance, if $\alpha=1$ and $\beta=7$, we have $\omega(\beta)=0<1=\omega(\alpha)$.  The full set is $$B_\omega=\{(1,7),(1,8),(6,11),(6,16),(9,16),(10,16),(10,20),(12,20)\}.$$ 
Notice that $|A_\omega|=|B_\omega|=8$.
\end{example}

\begin{lem}
\label{lem: A size}
Suppose that $\omega^{-1}=\omega_{\pi}$ is a $\gamma$-restricted permutation corresponding to the $(K,k)$ parking function $\pi$ by Lemma \ref{lem: pf bijection}. Then  
$$
|A_{\omega}|=|B_\omega|=\inv(\omega^{-1})-(\delta_{K,k}-\dinv(\pi)).
$$
\end{lem}

\begin{proof}
Recall that by Lemma \ref{lem:deltaminusdinv} we have
\[
\delta_{K,k}-\dinv(\pi)=|\{(\alpha,\beta):1\leq \omega(\alpha)\leq K, \omega(\beta)\le \omega(\alpha), \alpha<\beta\leq \alpha+k\}|.
\]
Therefore
\begin{align*}
 \inv(\omega^{-1}) - (\delta_{K,k}-\dinv(\pi))&=
|\{(\alpha,\beta) \mid 1\leq \omega(\alpha) \leq K,\, \omega(\beta)<\omega(\alpha),\,\alpha<\beta\}| 
\\ &\phantom{=.}-|\{(\alpha,\beta) \mid 1\leq \omega(\alpha)\leq K,\, \omega(\beta)<\omega(\alpha),\,\alpha<\beta\leq \alpha+k\}|
\\ &=|\{(\alpha,\beta)\mid 1\leq \omega(\alpha)\leq K,\,\omega(\beta)<\omega(\alpha),\,\alpha+k<\beta\}|\\
&=|B_{\omega}|.
\end{align*}
It now suffices to show $|A_\omega|=|B_\omega|$.  We construct a bijection $\varphi:A_\omega\to B_\omega$ by defining $\varphi(\alpha,\beta)=(\alpha,\gamma(\beta))$.
Since $\gamma$ is invertible, it suffices to show $(\alpha,\beta)\in A_{\omega}$ if and only if $(\alpha,\gamma(\beta))\in B_{\omega}$.

    Suppose $(\alpha,\beta)\in A_{\omega}$. Then $1\leq \omega(\alpha)\leq K$, $\alpha<\beta$, and by the $\gamma$-restricted condition we have $\omega(\beta)<\omega(\gamma(\beta))<\omega(\alpha)$. Thus, it remains to show that $\alpha+k < \gamma(\beta)$. 
    Indeed, since $\alpha<\beta$, then $\alpha+k < \beta+k\leq \gamma(\beta)$, so $\alpha+k < \gamma(\beta)$, and thus $(\alpha,\gamma(\beta))\in B_\omega$.

    Suppose $(\alpha,\gamma(\beta))\in B_{\omega}$. Then $1\leq \omega(\alpha)\leq K$, $\omega(\gamma(\beta)) < \omega(\alpha)$, and $\alpha+k < \gamma(\beta)$. It suffices to show $\alpha < \beta$. We have the following cases:
    
    \textbf{Case 1.}  If $\gamma(\beta)=\beta+k$, then $\alpha + k <\beta+k$, so $\alpha<\beta$ and we are done.

    \textbf{Case 2.} $\gamma(\beta)=\beta+k+1$, then $\alpha + k <\beta+k+1$ and $\alpha\le \beta$. But $\omega(\gamma(\alpha))>\omega(\alpha)>\omega(\gamma(\beta))$, so $\alpha\neq \beta$ and we are done.

    \textbf{Case 3.} Finally, suppose  $\beta=mK$. Since $1\le \omega(\alpha)\le K$, we have $\alpha>0$ and $\gamma(\beta)>\alpha+k>k$, while $\omega(\gamma(\beta)) <\omega(\alpha)\le K$. By Corollary \ref{cor: beta mK} we get a contradiction. 
\end{proof}

We will require one more technical lemma about the elements of the set $A_\omega$.

\begin{lem}
\label{lem:A-properties}
Suppose that $\omega^{-1}$ is $\gamma$-restricted. Then:
\begin{enumerate}
\item  For all $(\alpha,\beta)\in A_{\omega}$ we have  
$
\omega(\beta)<\omega(\alpha)$ and $\omega(\gamma(\beta))<\omega(\gamma(\alpha)).
$

\item  Assume $1\leq \omega(\alpha)\leq K, \gamma(\alpha) < \gamma(\beta)$, and $\omega(\gamma(\beta)) < \omega(\alpha)$. Then $\alpha<\beta$, and $(\alpha,\beta)\in A_{\omega}$.
\end{enumerate}
\end{lem}

\begin{proof}
For (1), by the $\gamma$-restricted condition, we have 
$$
\omega(\beta)<\omega(\gamma(\beta))<\omega(\alpha)<\omega(\gamma(\alpha)).$$  Above, the middle inequality follows from the definition of $A_\omega$.

For (2), since $1\le \omega(\alpha)\le K$, we have $\alpha>0$ since $\omega^{-1}$ is $\gamma$-restricted. 
 We also have, by the assumptions and the definition of $\gamma$, that $$\gamma(\beta)>\gamma(\alpha)\ge \alpha+k>k$$ since $\alpha>0$.  On the other hand, $\omega(\gamma(\beta))<\omega(\alpha)\le K$, so by Corollary \ref{cor: beta mK} we get $\beta\neq mK$.
By Lemma \ref{lem: gamma almost increasing}, we conclude that $\alpha<\beta$, and so $(\alpha,\beta)\in A_\omega$.
\end{proof}

\section{Affine $\Delta$-Springer fibers}
\label{sec: geometry}

In this section, we define the varieties $X_{n,k}$ and $Y_{n,k}$. We then show that they have Springer actions which are related by a Schur skewing operator.

\subsection{Affine flags}

Let $n\geq k$ be two positive integers and define $K = k(n-k+1)$ as above. Let $\bO = \bC[[\epsilon]]$ and  $\bK = \bC((\epsilon))$ be the rings of formal power series and Laurent series, respectively. Given an element 
$$
f=a_{m}\epsilon^m+a_{m+1}\epsilon^{m+1}+\cdots\in \bK,\quad \text{ with } a_m\neq 0
$$
we define its valuation as $\nu(f)=m$.

We will consider the loop groups $\GL_K(\bO)\subset \GL_K(\bK)$ together with the evaluation map
$$
\ev: \GL_K(\bO)\to \GL_K(\bC)
$$ defined by setting $\epsilon\mapsto 0$.
Let $I = \ev^{-1}(B)$ be the usual \textbf{Iwahori subgroup} of $\GL_K(\bO)$ where $B$ is the Borel subgroup of $\GL_K(\bC)$ of upper triangular matrices, and let $I^- = \ev^{-1}(B^-)$ the opposite Iwahori subgroup where $B^-$ is the subgroup of lower triangular matrices.

More generally, for any composition $\eta\vDash K$ let $\bP_\eta = \ev^{-1}(P_\eta)$ be the \textbf{parahoric subgroup} of type $\eta$, where $P_\eta\subseteq \GL_K(\bC)$ is the parabolic subgroup consisting of block-lower-triangular matrices with block sizes $\eta_i$. In particular, $\bP_{(1^K)}=I^-$ where $I^-$ is the Iwahori group defined above, $\bP_{(K)}=\GL_K(\bO)$, and $\bP_{(K-n,1^n)} = \ev^{-1}(P_{(K-n,1^n)})$ is the parahoric subgroup corresponding to the parabolic subgroup for the composition $(K-n,1^n)$.

\begin{defn}
The affine flag variety is defined as $\aFl=\aFl_{(1^K)} = \GL_K(\bK)/I^{-}$ and the affine Grassmannian is defined as 
$\aGr=\aFl_{(K)} = \GL_K(\bK)/\GL_K(\bO)$. More generally, for a composition $\alpha$ we define the partial affine flag variety
$$
\aFl_{\eta}=\GL_K(\bK)/\bP_{\eta}.
$$
We have natural projection maps
\begin{equation}
\label{eq: pr eta}
\aFl\xrightarrow{\pr_{\eta}} \aFl_{\eta}\to \aGr.
\end{equation}
\end{defn}

Equivalently, we can visualize the points in $\aGr$ using lattices. A lattice is a free $\bO$-submodule in $\bK^K$ of rank $K$. The affine Grassmannian is isomorphic to the space of $\bO$-lattices $\Lambda$ in $\bK^K$ by identifying $g\in \GL_K(\bK)/\GL_K(\bO)$ with the lattice $\Lambda_g=g\bO^K$. Note that $\GL_K(\bK)$ acts transitively on the set of all lattices, and the stabilizer of $\bO^K$ is the group $\GL_K(\bO)$.

Similarly, the points in $\aFl$ correspond to flags of lattices 
$$
\Lambda_{\bullet}=(\Lambda_0\supset \Lambda_1\supset \cdots\supset \Lambda_K= \epsilon \Lambda_0)
$$
where $\Lambda_i=g\Lambda_i^{\std}$ and 
$$
\Lambda_i^{\std}=\bO\{\epsilon e_1,
\ldots, \epsilon e_{i},e_{i+1},\ldots,e_K\}.
$$
The stabilizer of the standard flag $\Lambda^{\std}_{\bullet}$ is precisely the group $I^{-}$. Sometimes it will be useful to  consider the lattices $\Lambda_i$ for all $i\in \bZ$ extending  periodically as $\Lambda_{i+K}=\epsilon \Lambda_i$.

For a composition $\eta=(\eta_1,\ldots,\eta_s)$ a point in the partial flag variety $\aFl_{\eta}$ is again a flag of lattices
$$
(\Lambda_0\supset \Lambda_{\eta_1}\supset \Lambda_{\eta_1+\eta_2}\supset\cdots\supset \epsilon \Lambda_0)
$$
where we can identify $\Lambda_i$ with  $g\Lambda_i^{\std}$.  The projection $\aFl\to \aFl_{\eta}$ forgets some lattices in the complete flag, in particular the projection  $\aFl\to \aGr$ sends $\Lambda_{\bullet}$ to $\Lambda_0$.





Let $T\cong (\bC^*)^K$ be the algebraic torus of diagonal $K\times K$ matrices with entries in $\bC^*$ which acts on $\aFl$. Then the $T$-fixed points are exactly $\omega I^-$ for $\omega\in \widetilde{S}_K$. Here we identified $\omega$ with the corresponding affine permutation matrix in $\GL_K(\bK)$. Using the decomposition $\omega=\ttt_{\lambda}\ww$ as in \eqref{eq: omega decomposition}, we can write $\omega I^- = \Lambda_\bullet$ where
\[
\Lambda_i =\bO\{ \epsilon^{\lambda_{\ww_1}+1}e_{\ww_1},\dots, \epsilon^{\lambda_{\ww_{i}}+1}e_{\ww_{i}},\epsilon^{\lambda_{\ww_{i+1}}} e_{\ww_{i+1}},\dots, \epsilon^{\lambda_{\ww_K}} e_{\ww_K} \},
\]
so in particular 
\[
\Lambda_{0} = \bO\{\epsilon^{\lambda_{\ww_1}}e_{\ww_1},\dots, \epsilon^{\lambda_{\ww_K}}e_{\ww_K}\} = \bO\{\epsilon^{\lambda_1}e_1,\dots, \epsilon^{\lambda_K}e_K\}.
\]

\begin{example}\label{ex:flag}
  Consider the affine permutation $\omega=[1,5,11,16,6,9,20,10,12,14,15,19]$ from our running examples above.  We have $\omega=\ttt_\lambda \ww$ where $\ww=1,5,11,4,6,9,8,10,12,2,3,7$ and $\lambda=(0,1,1,1,0,0,1,1,0,0,0,0)$.  Then $\omega I^{-}$ corresponds to the flag $\Lambda$ with $$\Lambda_0=\bO\{e_1,e_5,e_{11},\epsilon e_4, e_6, e_9, \epsilon e_{8}, e_{10}, e_{12}, \epsilon e_2, \epsilon e_3, \epsilon e_7\}$$ and, for instance, $$\Lambda_{7}=\bO\{\epsilon e_1,\epsilon e_5,\epsilon e_{11},\epsilon^2 e_4, \epsilon e_6, \epsilon e_9, \epsilon^2 e_{8}, e_{10}, e_{12}, \epsilon e_2, \epsilon e_3, \epsilon e_7\}$$ and $$\Lambda_8=\bO\{\epsilon e_1,\epsilon e_5,\epsilon e_{11},\epsilon^2 e_4, \epsilon e_6, \epsilon e_9, \epsilon^2 e_{8}, \epsilon e_{10}, e_{12}, \epsilon e_2, \epsilon e_3, \epsilon e_7\}.$$  In particular, note that $\Lambda_7/\Lambda_8$ is spanned by $e_{10}=e_{\omega(8)}$.
\end{example}

\subsection{Affine Schubert cells}

We now define the Schubert decomposition of the affine flag variety.

\begin{defn}
Given $\omega\in \widetilde{S}_K$, we define the Schubert cell by $C_{\omega}:=I^{-}\omega I^{-}\subset \aFl$. 
\end{defn}
It is known  that $\bigsqcup_{\omega}C_{\omega}=\aFl$. To describe the Schubert cells in terms of lattices, we parametrize the (topological) $\bC$-basis $e_i,\ i\in \bZ$ in $\bK^K$ by extending $e_1,\ldots,e_K$ periodically via $e_{i+K}=\epsilon e_i$. Then $\Lambda_{\bullet}$ is in $C_\omega$ if and only if
$\Lambda_{i-1}/\Lambda_{i}$ is spanned by
$$
f_{\omega(i)}=e_{\omega(i)}+\cdots
$$
where the higher degree terms are in $\bC\{e_m: m>\omega(i)\}$.

The above expansion is not unique since we can add to $f_{\omega(i)}$ a linear combination of $f_{\omega(j)}$ whenever $j>i$ and $\omega(j)>\omega(i)$. To sum up, we get a  unique {\bf canonical generator}
\begin{equation}
\label{eq: canonical}
f_{\omega(i)}=e_{\omega(i)}+\sum_{\stackrel{j<i}{\omega(j)>\omega(i)}}\lambda_{\omega(j)}^{\omega(i)}e_{\omega(j)}\in \Lambda_{i-1}.
\end{equation}

\begin{example}
    Continuing from Example \ref{ex:flag}, another element of $C_\omega$ is obtained by multiplying $\Lambda_\bullet$ by the element $M\in I_-$ that sends each $e_i\mapsto e_i$ for $i\neq 8$, and sends $e_{10}\mapsto e_{10}+e_{12}+2\epsilon e_4$.  Note that the latter may be written as $e_{10}\mapsto e_{10}+e_{12}+2e_{16}$ in our new periodic notation. Then $M\Lambda_{7}/M\Lambda_8$ is spanned by $$f_{\omega(8)}=f_{10}= e_{10}+e_{12}+2 e_{16}=e_{\omega(8)} + e_{\omega(9)} + 2 e_{\omega(4)}.$$

    The canonical representative then would only include the $e_{\omega(4)}$ term since $4<8$, so the canonical generator is $f_{\omega(8)}=e_{\omega(8)}+2e_{\omega(4)}$, and we have $\lambda_{16}^{10}=2$.
\end{example}

We conclude that  $C_{\omega}$ is isomorphic to the affine space of dimension
\begin{equation}\label{eq:dim-Cw}
\dim C_{\omega}=\inv(\omega)=|\{(i,j)\ :\ j<i,\, 1\le i\le K,\, \omega(j)>\omega(i)\}|.
\end{equation}
Also, note that with the above notations
\begin{equation}
\label{eq: lambda i span}
\Lambda_i=\bC\{f_{\omega(j)}:j > i\},
\end{equation}
where this notation means the formal $\bC$-span allowing infinitely many terms.

\begin{rmk}
\label{rem: components}
Since we work with $\GL_K$ instead of $\mathrm{SL}_K$, the affine Grassmannian and the affine flag variety both have infinitely many connected components, one for each integer $m$ which we call $\aFl(m)$. These can be characterized in terms of Schubert cells as 
$$
\aFl(m)\coloneqq \bigsqcup_{\omega:\deg(\omega)=m}C_{\omega}.
$$
Equivalently, a point in $\aFl(m)$ corresponds to $g\in \GL_K(\bK)$ such that $\nu(\det(g))=m$ (see \cite[Theorem 1.3.11]{Zhu} and reference therein).
\end{rmk}

The following is standard.

\begin{lem}
Assume $\omega\in \widetilde{S_K}^+$ is a positive affine permutation and $\Lambda_{\bullet}\in C_\omega$. Then $\Lambda_0\subset \bO^K$ and 
$$
\dim \bO^K/\Lambda_0=\deg(\omega).
$$
\end{lem}

We can also define Schubert cells in partial affine  flag varieties. The cells in $\aFl_{\eta}$ are labeled by cosets $\widetilde{S_K}/S_{\eta}$ where $S_{\eta}=S_{\eta_1}\times \cdots\times S_{\eta_{\ell}}$ is the parabolic subgroup in $S_K$. More precisely, since the canonical generators above are preserved under projection, we have the following. 

\begin{prop}
\label{prop: schubert cells in partial flags}
Let $\pr_{\eta}:\aFl\to \aFl_{\eta}$ be the projection defined in \eqref{eq: pr eta}. Suppose that $\omega\in S_K$ is the minimal  length (that is, minimal $\inv(\omega)$) representative in its coset in $\widetilde{S_K}/S_{\eta}$, and $C_{\omega}$ the corresponding Schubert cell in $\aFl$. Then $\pr_{\eta}:C_{\omega}\to \pr_{\eta}(C_{\omega})$ is an isomorphism and $\pr_{\eta}(C_{\omega})$ is the Schubert cell in $\aFl_{\eta}$.
\end{prop}

Using decomposition \eqref{eq: omega decomposition}, we can also parametrize $\widetilde{S_K}/S_{\eta}$ as the set of $\ttt_{\lambda}\ww$ where $\ww\in S_K/S_{\eta}$. In particular, for $\eta=(K)$ the Schubert cells in $\aGr$ are parametrized by $\ttt_{\lambda}$ for $\lambda\in \bZ^K$. 

\subsection{The varieties $X_{n,k}$ and $Y_{n,k}$} Next, we define some subvarieties in (partial) affine flag varieties. These depend on the affine permutation $\gamma=\gamma_{n,k,N}$ from \eqref{eq: def gamma} which we upgrade to an operator on $\bO^K$ but denote by the same letter.
 
\begin{defn}
    Let $\gamma=\gamma_{n,k,N}$ be the $\bO$-linear operator on $\bO^K = \bC^K[[\epsilon]]$ defined by
    \[
    \gamma e_i = e_{\gamma(i)}=
    \begin{cases}
    e_{i+k} & \text{if }1\leq i\leq (n-k)k\\
    e_{i+k+1} & \text{if } (n-k)k < i < K\\
    \epsilon^{N+1} e_1 & \text{if }i=K.
    \end{cases}
    \]
\end{defn}

By Lemma \ref{lem: gamma is a permutation}, $\gamma$ corresponds to an affine permutation, and in particular, it is invertible. Moreover, the characteristic polynomial of $\gamma$ equals
$$
\det(z I-\gamma)=z^{K}-\epsilon^{N+k}
$$
and all eigenvalues of $\gamma$ are equal to $\zeta \epsilon^{\frac{N+k}{K}}$ where $\zeta$ runs over $K$-th roots of unity. In particular, $\gamma$ is regular, semisimple and equivalued.

\begin{defn}
The affine Springer fiber in the partial affine flag variety is defined as
\[
 \Sp_{\gamma,\eta} \coloneqq\{g\bP_\eta\in \GL_K(\bK)/\bP_\eta \mid \Ad(g^{-1})\gamma \in \Lie(\bP_\eta)\}\subset \aFl_{\eta},
 \]
 where $\Ad(g^{-1})\gamma = g^{-1}\gamma g$.

 Define also $\Sp_\gamma = \Sp_{\gamma,(1^K)}$ and $\Gr_{\gamma} = \Sp_{\gamma,(K)}$.
 Equivalently, we can state this definition in terms of lattices as follows
\[
 \Sp_{\gamma,\eta} \coloneqq\{\Lambda_{\bullet}\in \aFl_{\eta}
 \mid \gamma \Lambda_i\subset \Lambda_i\ \mathrm{for\ all}\ i\}.
 \]
\end{defn}

\begin{defn}
We define the \textbf{affine Borho--MacPherson variety}
\[
\BM_{\gamma,n,k}\coloneqq\{\Lambda_\bullet \in \aFl_{(K-n,1^n)} \mid \gamma \Lambda_i\subseteq \Lambda_i\,\forall i,\, \JT(\gamma|_{\Lambda_{0}/\Lambda_{K-n}}) \leq (n-k)^{k-1}\}\subset \Sp_{\gamma,(K-n,1^n)}
\]
Here $\JT(\gamma|_{\Lambda_{0}/\Lambda_{K-n}})$ denotes the Jordan type of the induced action of $\gamma$ on $\Lambda_0/\Lambda_{K-n}$, and $\le$ denotes dominance order on partitions, so the largest blocks are at most size $n-k$. 
\end{defn}
We have the following two equivalent definitions of $\BM_{\gamma,n,k}$, where $\bP = \bP_{(K-n,1^n)}$.
\begin{align*}
\BM_{\gamma,n,k}&= \{\Lambda_\bullet \in \aFl_{(K-n,1^n)} \mid \gamma \Lambda_i\subseteq \Lambda_i\,\forall i,\, \gamma^{n-k}\Lambda_0\subseteq \Lambda_{K-n}\}\\
&= \{g\bP\in \GL_K(\bK)/\bP \mid \Ad(g^{-1})\gamma \in \Lie(\bP), \JT(\ev(\Ad(g^{-1})\gamma))\le (n-k)^{k-1}\}
\end{align*}
where $\bP=\bP_{(K-n,1^n)}$.


Next, we define some ``positive" versions of the affine flag variety by considering certain unions of Schubert cells. First, recall the definitions of positive (Definition \ref{def: positive permutation}) and normalized (Definition \ref{def: normalized permutation}) affine permutations. As above, define $\widetilde{S}_{K}^{+,0}$ to be the set of positive and normalized affine permutations.

Define the following unions of Schubert cells:
\begin{align}\label{eq:def-of-C}
C &=  \bigcup_{\omega\in \widetilde{S}_K^{+,0}} I^- \omega I^-/I^-\subset \aFl\\
C' &= \pr_{(K-n,1^n)}(C)\subset \aFl_{(K-n,1^n)}.
\end{align}
 


\begin{defn}\label{def:Varieties}
We define the varieties
\begin{align*}
X_{n,k,N}&\coloneqq \Sp_{\gamma}\cap C\subset \aFl,\\
Y_{n,k,N}&\coloneqq  \BM_{\gamma,k,n}\cap C'\subset \aFl_{(K-n,1^n)}.
\end{align*}
\end{defn}



\begin{lem}
\label{lem: schubert cell intersects springer}
The intersection of a Schubert cell $I^-\omega I^-/I^-$ with $\Sp_{\gamma}$ is nonempty if and only if $\omega^{-1}(\gamma(a))>\omega^{-1}(a)$ for all $a$.
If it is nonempty, it contains the fixed point $\omega$.
\end{lem}

\begin{proof}
A flag $\Lambda_{\bullet}$ in the Schubert cell is in $\Sp_{\gamma}$ if and only if $\gamma \Lambda_{i-1}\subset \Lambda_{i-1}$ for all $i$. By \eqref{eq: lambda i span}, this is equivalent to 
$$
\gamma f_{\omega(i)} \in \bC\{f_{\omega(j)}:\ j \ge i\}.
$$
Since the leading term in $f_{\omega(i)}$ is $e_{\omega(i)}$, the leading term in $\gamma f_{\omega(i)}$ is $e_{\gamma\omega(i)}$, and we need
$$
\gamma(\omega(i))=\omega(j)\ \text{for some}\ j>i.
$$
By denoting $a=\omega(i)$, we get 
$$
\omega^{-1}(\gamma(a))>\omega^{-1}(a).
$$
Conversely, if this inequality is satisfied for $\omega$ then the above computation shows that the fixed point $\omega$ is in $\Sp_{\gamma}$.
\end{proof}

 We recall the definition of a $\gamma$-restricted permutation (Definition \ref{def: gamma restricted}).

\begin{cor}
\label{cor: cells Y}
The intersection of a Schubert cell $I^-\omega I^-/I^-$ with $X_{n,k,N}$ is nonempty if and only if $\omega$ is $\gamma$-restricted.
If it is nonempty, it contains the $T$-fixed point $\omega I^-$.
\end{cor}

\begin{proof}
Indeed, by Lemma \ref{lem: schubert cell intersects springer} $\omega$ needs to satisfy 
$\omega^{-1}(\gamma(x))>\omega^{-1}(x)$ for all $x$, and by definition of $C$, $\omega$ needs to be positive and normalized. 
\end{proof}






\begin{figure}[ht!]
    \centering
    \includegraphics{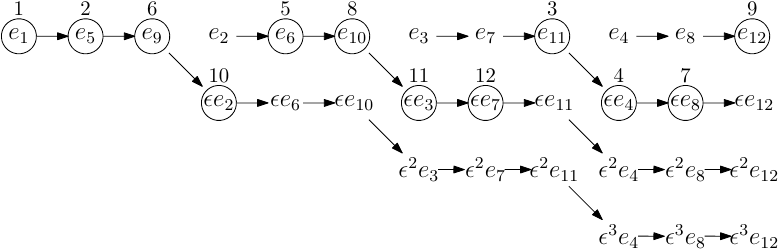}
    \caption{The staircase diagram of the fixed point $\epsilon^\lambda \ww I^-$ in $X_{n,k,N}=\Sp_\gamma\cap C$ in the case when $n=6$, $k=4$, $K=12$, $N$ is arbitrary and where $\lambda = (0,1,1,1,0,0,1,1,0,0,0,0)$ and $w = [1,5,11,4,6,9,8,10,12,2,3,7]$.}
    \label{fig:staircase-diagram}
\end{figure}

One can visualize the $T$-fixed point of 
$X_{n,k,N}$ as follows: List the vectors $\epsilon^j e_i$ for $0\leq j\leq (i-1)\bmod k$ with a directed arrow from $\epsilon^je_i$ to $\gamma(\epsilon^je_i)$. Given a $T$-fixed point $\epsilon^\lambda \ww I^{-}$, circle the vector $\epsilon^{\lambda_i}e_i$ and label it by $\ww^{-1}(i)$. See Figure~\ref{fig:staircase-diagram} for an example of a staircase diagram.

Following Lemma \ref{lem: pf bijection} we can also associate a $(K,k)$ parking function to a $\gamma$-restricted permutation $\omega$. It is easy to see that it precisely matches the staircase diagram up to $90^{\circ}$ rotation, compare Figures \ref{fig:parking-func-of-w} and \ref{fig:staircase-diagram}.

\begin{lem}
The $T$-fixed points in $Y_{n,k,N}$ are in bijection with $\LD_{n,k}^{\stack}$. In particular, the set of $T$-fixed points does not depend on $N$.
\end{lem}

\begin{proof}
Similarly to Corollary \ref{cor: cells Y}, the fixed points in $C'\cap \Sp_{\gamma,(K-n,1^n)}$ are in bijection  with the cosets of $\gamma$-restricted affine permutations by the right action of $S_{K-n}.$ Following Lemma \ref{lem: pf bijection}, we can identify these with the cosets of $(K,k)$-parking functions by the action of $S_{K-n}$ which permutes the labels $n+1,\ldots,K$. Equivalently, we can just call all these ``big labels" and not distinguish them, as in Section \ref{sec: stacks}.

To describe fixed points in $Y_{n,k}$, we need to further constrain the Jordan type of $\gamma$ acting on $\Lambda_{0}/\Lambda_{K-n}$. Suppose $\Lambda_\bullet\in Y_{n,k}^T$. Then the action of $\gamma$ on $\Lambda_{0}/\Lambda_{K-n}$ has Jordan blocks of size $b_i$ where $b_i$ is the number of big labels in the $i$-th column. We have $\JT(\Lambda_{0}/\Lambda_{K-n})\le (n-k)^{k-1}$ if and only if $b_i\le n-k$ for all $i$.  

By Lemma \ref{lem: admissible to stacks}, such parking functions (up to permutation of the big labels) are in bijection with $\LD_{n,k}^{\stack}$.
\end{proof}

\subsection{Stabilization}

\begin{defn}
We define the \textbf{conductor} $c$ by the equation $c=k\cdot K$. 
\end{defn}

Note that $c$ is the index of the largest basis vector in the staircase diagram. For example, in Figure \ref{fig:staircase-diagram} the largest basis vector is $\epsilon^3 e_{12}=e_{48}$, so $c=48$. 

\begin{lem}
\label{lem: conductor is super effective}
For all flags $\Lambda_{\bullet}\in X_{n,k,N}$ and all $m\ge c$ we have $e_m\in \Lambda_0$. The same holds for all partial flags $\Lambda_{\bullet}\in Y_{n,k,N}$.
\end{lem}

\begin{proof}
Let us first prove that for all $m\ge c$ there exists a vector $\widetilde{f_m}=e_m+\cdots\in \Lambda_0$.

Assume $\Lambda_{\bullet}\in C_{\omega}$ for some Schubert cell $C_{\omega}$. By Corollary \ref{cor: cells Y}, $\omega$ is $\gamma$-restricted. Note that $\omega(1),\ldots,\omega(K)$ all have different remainders modulo $K$, and by \ref{lem: A Wild Conductor Appears} $\omega(i)\le c$ for $i=1,\ldots,K$. Therefore for all $m\ge c$ there exists $a=\omega(i)\le c$ such that $a=m\mod K$ and there exists a vector 
$$
f_a=e_a+\cdots\in \Lambda_{i-1}\subset \Lambda_0.
$$
Now we can choose $\widetilde{f_m}=\epsilon^{(m-a)/K}f_a=e_m+\cdots\in \Lambda_0$. Finally, for any $m\ge c$ we can start with $\widetilde{f_m}$ and subtract a linear combination of  $\widetilde{f_{m'}}$ for  $m'>m$ to eliminate all higher degree coefficients. 

Since any partial flag in $Y_{n,k,N}$ can be extended to at least one flag in $X_{n,k,N}$,  the above argument applies to $Y_{n,k,N}$ verbatim.
\end{proof}

\begin{lem}
\label{lem: N stability}
For $N\ge k$ the spaces $X_{n,k,N}$ and $Y_{n,k,N}$ do not depend on $N$. That is, for $N,N'\ge k$ we have 
$$
X_{n,k,N}=X_{n,k,N'},\quad Y_{n,k,N}=Y_{n,k,N'}.
$$
In particular, for $N,N'\ge k$  a flag in $C$ is invariant under $\gamma_{n,k,N}$ if and only if it is invariant under $\gamma_{n,k,N'}$. 
\end{lem}

\begin{proof}Note that 
$$
\epsilon^k\Lambda_0^{\std}=\bC\{e_m:m\ge k\cdot K+1\}=
\bC\{e_m:m\ge c+1\}.
$$
By Lemma \ref{lem: conductor is super effective} for all $\Lambda_{\bullet}$ in $X_{n,k,N}$ we have $\epsilon^k\Lambda_0^{\std}\subset \Lambda_0$, so $\epsilon^{k+1}\Lambda_0^{\std}\subset \epsilon \Lambda_0$ for all $\Lambda_{\bullet}\in X_{n,k}$. 
In particular, 
$$
\Lambda_0^{\std}\supset \Lambda_i\supset \epsilon^{k+1}\Lambda_0^{\std}\ \mathrm{for\ all}\  i=0,\ldots,K-1.
$$

Given a flag $\Lambda_{\bullet}\in X_{n,k,N}$, we can define the subspace $\overline{\Lambda_i}=\Lambda_i/\epsilon^{k+1}\Lambda_0^{\std}$ in the finite-dimensional space $\Lambda_0^{\std}/\epsilon^{k+1}\Lambda_0^{\std}$.
This yields an isomorphism  between $X_{n,k,N}$ and the space $\overline{X_{n,k,N}}$  of flags $\overline{\Lambda_{\bullet}}$ which are invariant under both $\epsilon$ and the restriction $\overline{\gamma}$ of $\gamma$ to $\Lambda_0^{\std}/\epsilon^{k+1}\Lambda_0^{\std}$, and are contained in the union of Schubert cells given by an appropriate truncation of $C$. 

On the other hand, $\gamma(e_K)=\epsilon^{N+1}e_{1}$ and for $N\ge k$ we have $\gamma(e_K)\in \epsilon^{k+1}\Lambda_0^{\std}$. Therefore for $N\ge k$ the restriction $\overline{\gamma}$ does not depend on $N$, and hence the truncated space $\overline{X_{n,k,N}}$ does not depend on $N$. 

Finally, the flag $\Lambda_{\bullet}\in X_{n,k,N}$ can be reconstructed from $\overline{\Lambda_{\bullet}}\in \overline{X_{n,k,N}}$ by simply taking the preimages of $\overline{\Lambda_i}$ under the projection 
$$\Lambda_0^{\std}\to \Lambda_0^{\std}/\epsilon^{k+1}\Lambda_0^{\std}
$$
and hence does not depend on $N$ either. The proof for $Y_{n,k,N}$ is similar.
\end{proof}

Motivated by Lemma \ref{lem: N stability}, we can define the ``stable'' version of the spaces $X_{n,k,N}$ and $Y_{n,k,N}$.

\begin{defn}
We define $X_{n,k}=X_{n,k,N}$ and $Y_{n,k}=Y_{n,k,N}$ where $N\ge k$ is an arbitrary integer. 
\end{defn}

\begin{rmk}
We can define the nilpotent operator $\gamma_{\infty}$ by
$$
\gamma_{\infty}(e_i)=\begin{cases}
    e_{i+k} & \text{if }1\leq i\leq (n-k)k\\
    e_{i+k+1} & \text{if } (n-k)k < i < K\\
    0 & \text{if }i=K.
    \end{cases}
$$
Then the proof of Lemma \ref{lem: N stability} implies that $X_{n,k}=\Sp_{\gamma_{\infty}}\cap C$.

It would be interesting to explore further this connection between the regular semisimple and nilpotent affine Springer fibers.
\end{rmk}

\subsection{Torus action}

The torus $T$ does not act on the affine Springer fiber $\Sp_{\gamma,\eta}$. However, we have an action of $\bC^*$ defined as follows. First, for $s\in \bC^*$ we define an operator $S$ on $\bK^K=\bC((\epsilon))^K$ as follows:
\begin{align}
S_N(\epsilon^{m}e_{i})&=
s^{\theta_N(m,i)}\epsilon^{m}e_{i}, \label{eq: def S}\\
\text{where }\theta_N(m,i)&=Km+q(N+k)+N(n-k+1)(r-1).
\end{align}

Here $1\le i=qk+r\le K$ where $0\le q\le n-k$ and $1\le r\le k$ and $m\in \bZ$.

\begin{example}
    For $k=4$, $K=12$, we have $\theta_N(2,6)=2N+(N+4)+N\cdot 3 \cdot 1 = 5N+4$, and so $$S_N(\epsilon^2 e_6)=s^{5N+4}\epsilon^2 e_6.$$
\end{example}

The motivation for our definition of $S_N$ above is the following lemma.

\begin{lem}
\label{lem: C* weights}
We have $S_N\epsilon=\diag(s^{K})\epsilon S_N$ and $S_N\gamma=\diag(s^{N+k})\gamma S_N$.
\end{lem}

\begin{proof}
For the first equation, it is easy to check that $S_N(\epsilon e_i)=s^K\epsilon S_N(e_i)$, and the equation follows.  For the second equation, it suffices to show $S_N\gamma e_i = \diag(s^{N+k})\gamma S_N e_i$ for $1\leq i\leq K$, or equivalently that $\theta_N(m',i') - \theta_N(0,i) = N+k$ where $\epsilon^{m'}e_{i'} = \gamma(e_i)$. We consider several cases:

\textbf{Case 1.} For $1\le i=qk+r\le (n-k)k$ we have $0\le q<n-k$, so $m'=0$ and $i'=i+k=(q+1)k+r$. Therefore,
$$
\theta_N(0,i+k)-\theta_N(0,i)=(q+1)(N+k)+N(n-k+1)(r-1) - \theta_N(0,i)=N+k.
$$

\textbf{Case 2.} For $(n-k)k<i=qk+r<K$ we have $q=n-k$ and $1\le r<k$, so $m'=1$ and $i'=i+k+1=(n-k)k+r+k+1=K+r+1$. Therefore,
$$
\theta_N(1,r+1)=K+N(n-k+1)r,\ \theta_N(0,i)=(N+k)(n-k)+N(n-k+1)(r-1),\\
$$
so
\begin{align*}
\theta_N(1,r+1)-\theta_N(0,i)&=K+N(n-k+1)-(N+K)(n-k)\\
&=k(n-k+1)+N(n-k+1)-(N+K)(n-k)=N+k.
\end{align*}

\textbf{Case 3.} Finally, for $i=K$ we have $q=n-k$ and $r=k$ so $m'=N+1$ and $i'=1$, and we have
$$
\theta_N(0,K)=(N+k)(n-k)+N(n-k+1)(k-1),\quad \theta_N(N+1,1)=(N+1)K,
$$
so
\begin{align*}    
\theta_N(N+1,1)-\theta_N(0,K)&=(N+1)K-N(n-k+1)(k-1)-(N+k)(n-k)\\
&=(N+1)k(n-k+1)-N(n-k+1)(k-1)-(N+k)(n-k)\\
&=(k+N)(n-k+1)-(N+k)(n-k)=N+k
\end{align*}
as desired.
\end{proof}

\begin{cor}
\label{cor: C* action}
The operator $S_N$ defines a $\bC^*$-action on the affine Springer fiber $\Sp_{\gamma}$, and on the varieties $X_{n,k,N}$ and $Y_{n,k,N}$.
\end{cor}

\begin{proof}
Let $\Lambda$ be a lattice. Since $\epsilon S_N\Lambda=\diag(s^{-K})S_N\epsilon \Lambda\subset S_N\Lambda$, we conclude that $S_N\Lambda$ is a lattice, in particular, $S_N$ acts on the affine Grassmannian and all partial affine flag varieties. 
Furthermore, it preserves each affine Schubert cell and hence preserves the unions of cells $C$ and $C'$.

Now suppose that $\Lambda$ is $\gamma$-invariant. Then $\gamma S_N\Lambda=\diag(s^{-N-K})S_N\gamma \Lambda\subset S_N\Lambda$, and $S_N\Lambda$ is $\gamma$-invariant as well, therefore $S_N$ preserves the affine Springer fiber $\Sp_{\gamma}$ and its partial affine flag versions. This implies that $S_N$ preserves $X_{n,k,N}=\Sp_{\gamma}\cap C$.

Finally, to prove that $S_N$ preserves $Y_{n,k,N}$ it is sufficient to prove that it preserves $\BM_{\gamma,k,n}$. Indeed, $S_N$ provides an isomorphism between the vector spaces $\Lambda_{0}/\Lambda_{K-n}$ and $S_N\Lambda_{0}/S_N\Lambda_{K-n}$ which exchanges the actions of $\gamma$ on these up to a nonzero scalar factor $s^{-N-K}$. Therefore the Jordan types of $\gamma$ restricted to $\Lambda_{0}/\Lambda_{K-n}$ and $S_N\Lambda_{0}/S_N\Lambda_{K-n}$ agree, and the result follows.
\end{proof}

\begin{rmk}
The above $\bC^*$ action cannot be obtained by restricting the $T$-action. Indeed, the latter commutes with multiplication by $\epsilon$ but the former does not. In fact, one can think of the $\bC^*$ action as the action of a subgroup of the extended torus $T\times \bC^*_{\hbar}$ where $\bC^*_{\hbar}$ acts by rescaling $\epsilon$ (``loop rotation"). 
\end{rmk}

In the stable limit, the $\bC^*$ action from Corollary \ref{cor: C* action} extends to a $\bC^*\times \bC^*$ action. For $i=mK+qk+r$ where $m\in \bZ,0\le q\le n-k$ and $1\le r\le k$ we write
\begin{equation}
\label{eq: theta stable}
\theta(i)=(\theta_0(i),\theta_{\infty}(i))=(mK+qk,q+(n-k+1)(r-1)).
\end{equation}
This defines a diagonal  action of $\bC^*\times \bC^*$ on $\bO^n$ by $(s_1,s_2)\cdot e_i= s_1^{\theta_0(i)}s_2^{\theta_{\infty}(i)}e_i$.
\begin{rmk}
The weights $\theta_N$ from \eqref{eq: def S} can be written as 
$
\theta_N = \theta_0+N\theta_{\infty}.
$
\end{rmk}
\begin{lem}
\label{lem: stable torus action}
The weights \eqref{eq: theta stable} define a $\bC^*\times \bC^*$ action on $X_{n,k}$ and $Y_{n,k}$.
\end{lem}

\begin{proof}
We follow the proof of Lemma \ref{lem: N stability} and work with the truncated space $\overline{X_{n,k,N}}$ of flags in $\Lambda_0^{\std}/\epsilon^{k+1}\Lambda_0^{\std}$
instead $(N\ge k)$. 

It suffices to show that when $\overline{\gamma}(e_i)\neq 0$ (that is, when $i$ is not divisible by $K$), then the weights \eqref{eq: theta stable} satisfy 
\begin{equation}
\label{eq: theta stable 2}
\theta(\gamma(i))=\theta(i)+(k,1), \hspace{1cm} \theta(K+i)=\theta(i)+(K,0).
\end{equation}
The second identity clearly holds when $i$ is not divisible by $K$. For the first identity, similarly to Lemma \ref{lem: C* weights} we have two cases:

\textbf{Case 1.} If $q<n-k$ then we get $\gamma(i)=i+k=mK+(q+1)k+r$, so
$$
\theta(\gamma(i))=(mK+(q+1)k,q+1+(n-k+1)(r-1))=\theta(i)+(k,1).
$$

\textbf{Case 2.} If $q=n-k$ and $r<k$, we get $\gamma(i)=i+k+1=(m+1)K+(r+1)$, so
\begin{align*}
\theta(\gamma(i))&=((m+1)K,(n-k+1)r),\\
\theta(i)&=(mK+(n-k)k,n-k+(n-k+1)(r-1)),
\end{align*}
so \eqref{eq: theta stable 2} is satisfied in this case as well. Finally, for $q=n-k$ and $r=k$ we get $\overline{\gamma}(e_i)=0$, so we do not need to check anything.  The rest of the argument follows the proof of Corollary \ref{cor: C* action}.
\end{proof}

\section{Springer theory}\label{sec: Springer action}

\subsection{Springer action}

In this section, we prove that the Borel-Moore homologies of $X_{n,k}$ and $Y_{n,k}$ have actions of $S_K$, respectively $S_n$. We then use Borho and MacPherson's theory of partial resolutions of nilpotent varieties~\cite{BM} to show that they are related by a Schur-skewing operator in Theorem~\ref{thm:perp homology}. Throughout this section, we follow the conventions in \emph{loc.~cit.}~for perverse sheaves.

Let us denote $G=\GL_K\bC$, let $B$ the Borel subgroup of lower triangular invertible matrices, and let $P=P_{(K-n,1^n)}$ is the parabolic subgroup of type $\alpha=(K-n,1^n)$ of block-lower-triangular invertible matrices. We denote by $\mathfrak{g},\mathfrak{b}$ and $\mathfrak{p}$ the corresponding Lie algebras.
Recall that we have an evaluation map 
$
\ev: \GL_K(\bO)\to G
$ 
and the  parahoric subgroup $\bP = \bP_\alpha = \bP_{(K-n,1^n)} = \ev^{-1}(P)$.

Let $L$ be the Levi subgroup of $P$, let $\mathcal{N}_L$ be its nilpotent cone, and let $\mathfrak{n}=[\mathfrak{b},\mathfrak{b}]$ and $\mathfrak{n}_P = [\mathfrak{p},\mathfrak{p}]$ be the nilradicals of $\mathfrak{b}$ and $\mathfrak{p}$, respectively. Then we have the following maps, all of which are semismall by~\cite{BM},
\[
\begin{tikzcd}[/tikz/column 1/.append style={anchor=base east}, column sep=0]
    G\times_B \mathfrak{n} & \eqqcolon\widetilde{\mathcal{N}}\arrow[d,"\eta"]\arrow[dd,"\pi", bend left=50]  \\
    G\times_P(\mathcal{N}_L+\mathfrak{n}_P) & \eqqcolon\widetilde{\mathcal{N}}^P  \arrow[d,"\xi"]  \\
    &\mathcal{N}.
 \end{tikzcd}
\]
We then have the following commutative diagram where each square is cartesian. Here, each of the maps $\ev$ takes $g\bP_\eta\in \Sp_{\gamma,\eta}$ to $[1,\ev(Ad(g^{-1})\gamma)]\in \widetilde{\mathcal{N}}^{P_\eta}$. Note that $Ad(g^{-1})\gamma\in Lie(\bP_\eta)$ by definition, but after taking $\ev$ it lands in $\mathcal{N}_L+\mathfrak{n}_{P_\eta}$ by the fact that $\gamma$ is a nil-elliptic operator. This gives a map that is well-defined up to the left (adjoint) action of $P_\eta$ on $\mathfrak{p}_\eta$ and is hence well-defined up to the left (adjoint) action of $G$. Thus, we get a commutative diagram where each square is cartesian,
 \begin{equation}
 \label{eq: Springer cartesian}
 \begin{tikzcd}
     \Sp_\gamma \arrow[r,"\ev"]\arrow[d] & \text{$[G\backslash\widetilde{\mathcal{N}}]$}\arrow[d,"\hat\eta"] & \widetilde{\mathcal{N}}\arrow[d,"\eta"]\arrow[dd,"\pi" {yshift=10pt}, bend left=60] \arrow[l]  \\
     \Sp_{\gamma,\alpha} \arrow[r,"\ev"] \arrow[d] & \text{$[G\backslash \widetilde{\mathcal{N}}^P]$}\arrow[d,"\hat\xi"] & \widetilde{\mathcal{N}}^P  \arrow[d,"\xi"] \arrow[l] \\
    \Gr_\gamma   \arrow[r,"\ev"] & \text{$[G\backslash \mathcal{N}]$} & \mathcal{N}. \arrow[l]
    \arrow[from=1-2,to=3-2, "\hat\pi" {yshift=10pt}, bend left=60,crossing over]
 \end{tikzcd}
 \end{equation}
 Here, we consider all spaces with their reduced induced scheme structure.  Here, the notation $[G\backslash \mathcal{N}]$ (and similar constructions) means the quotient stack formed by the left action of $G$ on $\mathcal{N}$.


 In terms of lattices, $\ev$ sends a flag of lattices $\Lambda_{\bullet}$ to the flag $F_{\bullet}$ defined by $F_i=\Lambda_0/\Lambda_{K-i}$ in the $K$-dimensional space $\Lambda_0/\epsilon\Lambda_0$, together with the operator $\overline{\gamma}=\gamma|_{\Lambda_0/\epsilon\Lambda_0}$. This pair in turn defines an element of $\widetilde{\mathcal{N}}^P$, up to the left action of $G$ translating the flag and conjugating $\overline{\gamma}$.

 Given $t\in \mathcal{N}_L$, let $\mathcal{O}_t\coloneqq \Ad(L).t$, the adjoint orbit of $t$ in $\Lie(L)$. Furthermore, let $y = (1,t)\in \widetilde{\mathcal{N}}^P$, and let $\mathcal{O}_y \coloneqq G\times_P(\mathcal{O}_t + \mathfrak{n}) \subseteq \widetilde{\mathcal{N}}^P$, which are the strata of $\widetilde{\mathcal{N}}^P$ considered by Borho and MacPherson~\cite{BM}. 

  \begin{rmk}
  The discussion in \cite[Section 2.7]{GG} yields the following description of $\overline{\cO_y}$: We have $\overline{\cO_{y}}=G\times_P(\overline{\mathcal{O}_t} + \mathfrak{n})$. Alternatively, the variety $\overline{\cO_y}$ is isomorphic to the space of pairs $(F_{\bullet},x)\in G/P\times \mathcal{N}$ such that $xF_i\subseteq F_i$ for all $i$ and
  $$
  \JT(x|_{F_i/F_{i+1}})\leq \JT(t_i)
  $$
  where $t=t_0+t_{K-n}+t_{K-n+1}+\cdots+t_{K}$ is the block decomposition of $t$ (where $t_i=0$ for $i>0$ in our case, since these blocks are size $1$).
  \end{rmk}


Let $W^P=N_G(L_P)/L_P$ be the \emph{relative Weyl group}. Combinatorially, for $P=P_{\alpha}$ the group $W^P$ permutes blocks of $\alpha$ of equal size.
For our choice of $\alpha=(K-n,1^n)$, we have $W^P=S_n$.

\begin{prop}\cite[Proposition 2.6]{BM}\label{prop:BM-action-full}
\label{prop: partial Weyl action}
The partial Weyl group $W^P$ acts by endomorphisms on $\IC(\xi'_*\bQ_{\widetilde{\mathfrak{p}^{rs}}}) \cong R\xi_*\bQ_{\widetilde{\mathfrak{p}}}$, where $\mathfrak{p}^{rs}\subseteq \mathfrak{p}$ is the subset of regular semisimple elements, $\widetilde{\mathfrak{p}^{rs}} = \xi^{-1}(\mathfrak{p}^{rs})$, and $\xi': \widetilde{\mathfrak{p}^{rs}}\to \mathfrak{p}^{rs}$ is the restriction of $\xi$. 
\end{prop}

Let $d_x=\dim\pi^{-1}(x)$ and $d_y = \dim \xi^{-1}(y)$.

\begin{prop}\label{prop:BM-action-restriction}
    The partial Weyl group $W^P$ action on $\IC(\xi'_*\bQ_{\widetilde{\mathfrak{p}^{rs}}}) \cong R\xi_*\bQ_{\widetilde{\mathfrak{p}}}$ restricts to a $W^P$ action on $R\xi^y_*\IC(\bQ_{\cO_y})$.
\end{prop}

\begin{proof}
    By \cite[Proposition 1.10]{BM}, the perverse sheaf $R\xi_*^y \IC(\bQ_{\cO_y})[-2d_y]$ is a direct summand of the semisimple object $R\xi_* \bQ_{\widetilde{\mathfrak{p}}}[\dim\mathfrak{p}]$ in the category of perverse sheaves, and Borho and MacPherson give the following explicit formula for its decomposition into simple objects. 
    
    Let $V_x \coloneqq H^{2d_x}(\mathcal{B}_x)=V_{\JT(x)}$ and 
    \[
    V_y \coloneqq H^{2d_y}(\eta^{-1}(y))\cong H^{2d_y}(\mathcal{B}(L)_t) \cong H^{2d_y}(\mathcal{B}_{t_{0}}\times \cdots \times \mathcal{B}_{t_{n}}).
    \]
    By the Springer correspondence, $V_x$ is the irreducible representation of $W$ corresponding to $\JT(x)$, and $V_y$ is the irreducible representation of $W_L$ (the Weyl group of $L$) corresponding to $\JT(t_0),\dots, \JT(t_{n})$. Let 
    \[
    V_x^y \coloneqq \mathrm{Hom}_{W_L}(V_y,V_x),
    \]
    so that $V_x^y\otimes V_y$ is the $V_y$-isotypic component of $V_x$ as a $W_L$-module. 
    Observe that $V_x^y$ naturally inherits the structure of a $W^P$ module. 
    By \cite[1.10, Equation 6]{BM}, we have
    \[
    R\xi^y_*\IC(\bQ_{\cO_y})[-2d_y] \cong \bigoplus Rj^x_*\IC(\bQ_{\cO_x})[-2d_x] \otimes V_x^y
    \]
    where the sum runs over adjoint orbits $\cO_x$ in $\mathcal{N}$. Then $V_x^y$ inherits the structure of a $W^P$ module, which induces a $W^P$ action on the left-hand side which is the restriction of the action of $W^P$ on $R\xi_*\bQ_{\widetilde{\mathfrak{p}}}$.
\end{proof}

 The subvariety $\overline{\mathcal{O}_y}$ of $\widetilde{\mathcal{N}}^P$ gives a substack $[G\backslash \mathcal{O}_y]$ of $[G\backslash \widetilde{\mathcal{N}}^P]$, and we have restricted maps $\hat\xi^y$ and $\xi^y$:
 \begin{equation}\label{eq:commutative-diag1}
    \begin{tikzcd}
    {[G\backslash \widetilde{\mathcal{N}}^P]}\arrow[dd,"\hat\xi"{yshift=-10pt}] &&\widetilde{\mathcal{N}}^P\arrow[ll]\arrow[dd,"\xi"{yshift=-10pt}]&\\
    &{[G\backslash \overline{\mathcal{O}_y}]}\arrow[lu,hookrightarrow,"\hat j^y"]\arrow[ld,"\hat\xi^y"]&& \overline{\mathcal{O}_y}\arrow[ll,crossing over]\arrow[lu,"j^y",hookrightarrow]\arrow[ld,"\xi^y"]\\
    {[G\backslash \mathcal{N}]}&&\mathcal{N}\arrow[ll]&
    \end{tikzcd}
 \end{equation}

We also define $$C''  = \pr_{(K)}(C)\subset \aGr,$$ and note that $C''$ consists of Schubert cells in $\aGr$ labeled by $\ttt_{\lambda}\in \widetilde{S_K}$ such that $\lambda_i\ge 0$ for all $i$ and $\lambda_1=0$. By Remark \ref{rem: right action preserves positive} we get 
$C=\pr^{-1}(C'')$. Also notice that a lattice $\Lambda$ is in $C''$ if and only if $\Lambda\subset \bO^K$ and $\Lambda$  contains a vector with leading term $e_1$.  

We then obtain the following cartesian diagrams, which parallel \eqref{eq: Springer cartesian}, from the definitions and the discussion above.

\begin{prop} For $P = P_\alpha$ where $\alpha = (K-n,1^n)$ and $t\in \mathcal{N}_L$ such that $t_0$ has Jordan type $(n-k)^{k-1}$:

a) The affine Borho-MacPherson variety fits into a cartesian square
\begin{equation}
\label{eq: commutative diag BM}
\begin{tikzcd} 
        \BM_{\gamma,n,k} \arrow[r,"\ev"]\arrow[d,"\pr"]& {[G\backslash\overline{\mathcal{O}_y}]}\arrow[d,"\hat\xi^y"]\\
        \Gr_{\gamma} \arrow[r,"\ev"] & {[G\backslash\mathcal{N}]}.
    \end{tikzcd}
 \end{equation}   
 
b) The spaces $X_{n,k}$ and $Y_{n,k}$ fit into the following cartesian squares: 
 \begin{equation}
 \label{eq:commutative-diag2}
    \begin{tikzcd}
        Y_{n,k} \arrow[r,"\ev"]\arrow[d,"\pr"]& {[G\backslash\overline{\mathcal{O}_y}]}\arrow[d,"\hat\xi^y"]\\
        C''\cap \Gr_{\gamma} \arrow[r,"\ev"] & {[G\backslash\mathcal{N}]}.
    \end{tikzcd}
\quad
    \begin{tikzcd}
        X_{n,k} \arrow[r,"\ev"]\arrow[d,"\pr"]& {[G\backslash\widetilde{\mathcal{N}}]}\arrow[d,"\hat\pi"]\\
        C''\cap \Gr_{\gamma} \arrow[r,"\ev"] & {[G\backslash\mathcal{N}]}.
    \end{tikzcd}
 \end{equation}

\end{prop}

\begin{rmk}
Strictly speaking, these maps and diagrams are defined for $X_{n,k,N}$ and $Y_{n,k,N}$ for some specific $N$. On the other hand, it is easy to see by Lemma \ref{lem: N stability} that for $N\ge k$ the evaluation map $\ev$ does not depend on $N$, and hence the same cartesian diagrams hold for $X_{n,k}$ and $Y_{n,k}$. 
\end{rmk}

\begin{lem}\label{lem: Springer action Y}
    There is a ``Springer'' action of $W=S_K$ on the Borel-Moore homology of $X_{n,k}$.
\end{lem}

\begin{proof}
We follow the proof sketch in \cite[Subsection 2.6.3]{Yun}. Let $\bD_{[G\backslash \widetilde{N}]}$ be the dualizing complex of $[G\backslash \widetilde{N}]$. By the cartesian diagram on the right side of~\eqref{eq:commutative-diag2}, the fact that $\hat\pi$ is proper, and the fact that $\pr$ in this cartesian square is a map of algebraic varieties, by proper base change we have
\begin{equation}\label{eq:BaseChange}
R\pr_*\bD_{X_{n,k}}=\ev^!R\hat\pi_*\bD_{[G\backslash\widetilde{\mathcal{N}}]}.
\end{equation}
By standard results in Springer theory~\cite{BM,Lusztig-Greenpoly}, there are canonical identifications
\[
\End_{D([G\backslash \mathcal{N}])}(R\hat\pi_*\bQ_{[G\backslash \widetilde{\mathcal{N}}]}) = \End_{D_G(\mathcal{N})}(R\pi_*\bQ_{\widetilde{\mathcal{N}}}) = \End_{D(\mathcal{N})}(R\pi_*\bQ_{\widetilde{\mathcal{N}}}) = \bQ[W],
\]
where $W\cong S_K$ is the Weyl group of $G$, and where $D(-)$ (respectively $D_G(-)$) stand for the (respectively $G$-equivariant) bounded derived category of constructible sheaves. Hence, applying the Verdier duality functor, see~\cite[Corollary 2.8.9]{AcharBook}, we also have
\[
\End_{D([G\backslash \mathcal{N}])}(R\hat\pi_*\bD_{[G\backslash \widetilde{\mathcal{N}}]}) = \bQ[W].
\]
By~\eqref{eq:BaseChange}, the functor $\ev^!$ gives a map
\[
\bQ[W]= \End_{D([G\backslash\mathcal{N}])}(R\hat\pi_*\bD_{[G\backslash\widetilde{\mathcal{N}}]})\to \End_{D(C''\cap \Gr_\gamma)}(R\pr_*\bD_{Y_{n,k}}),
\]
and hence we also get a $W\cong S_K$ action on 
\[
R^{-i}\Gamma(\pr_*\bD_{X_{n,k}}) = R^{-i}\Gamma(\bD_{X_{n,k}})= H_{i}^{BM}(X_{n,k};\bQ),
\]
as desired.
\end{proof}

\begin{rmk}
    In order to conclude that there is a $S_n\cong W^P$ action on the Borel-Moore homology of $Y_{n,k}$, we could first show that $\overline{\mathcal{O}_y}$ is rationally smooth on all fibers of $\hat \xi^y$ over $ev(\Gr_\gamma)$. Unfortunately, this is not true. However, $\overline{\mathcal{O}_y}$ is rationally smooth on all fibers of $\hat\xi^y$ over $ev(C''\cap \Gr_\gamma)$, which we show in Subsection \ref{sec:GeomSkew}, which is exactly what we need to transfer the $S_n$ action to the Borel-Moore homology of $Y_{n,k}$.
\end{rmk}

\begin{lem}
\label{lem: parabolic invariants}
Given $\eta\vDash K$, consider the parabolic subgroup $S_{\eta}\subset S_K$ and the partial affine flag variety $\aFl_{\eta}$. Recall the projection $\pr_{\eta}:\aFl\to \aFl_{\eta}$. Then we have
$$
H^{\BM}_*(X_{n,k})^{S_\eta}\cong H^{\BM}_*(\pr_{\eta}(X_{n,k})),
$$
where the superscript $S_\eta$ denotes taking invariants with respect to the $S_\eta$ action.
\end{lem}

\begin{proof}
Let $P_\eta$ be the parabolic subgroup of $\GL_K\bC$ associated to $\aFl_\eta$, then $S_\eta\cong W_L$ is the Weyl group of the Levi subgroup $L$ of $P_\eta$. By \cite[Proposition 2.7b]{BM}, letting $\xi : \widetilde{\mathcal{N}}^{P_\eta} \to \mathcal{N}$ be the projection map as above,
\[
R\xi_*\bQ_{\widetilde{\mathcal{N}}^{P_\eta}} \cong (R\pi_*\bQ_{\widetilde{\mathcal{N}}})^{S_\eta}.
\]
This isomorphism descends to
\[
R\xi_*\bQ_{[G\backslash \widetilde{\mathcal{N}}^{P_\eta}]} \cong (R\pi_*\bQ_{[G\backslash\widetilde{\mathcal{N}}]})^{S_\eta},
\]
Applying the functor $\ev^*$ to both sides and then proper base change,
\[
R\pr_*\bQ_{\pr_\eta(X_{n,k})} \cong (R\pr_*\bQ_{X_{n,k}})^{S_\eta}.
\]
Applying Verdier dual and then derived global sections yields the result.
\end{proof}

\subsection{Geometric skewing formula}\label{sec:GeomSkew}

Lemma \ref{lem: Springer action Y} allows us to define the graded Frobenius character of the Borel-Moore homology of $X_{n,k}$ by considering the corresponding  representation of $S_K$ bigraded by homological degree and the connected component (see Remark \ref{rem: components}). In Theorem~\ref{thm:perp homology} below, we show that there is an action of $S_n$ on the Borel-Moore homology of $Y_{n,k}$, which allows us to define its graded Frobenius character. We will denote these by 
$$
\Frob_{q,t}H_*^{\BM}(X_{n,k})\ \mathrm{and}\ \Frob_{q,t} H_*^{\BM}(Y_{n,k}).
$$
The main result of this subsection is Theorem~\ref{thm:perp homology}, which relates these two Frobenius characters by the Schur skewing operator $s_{(n-k)^{k-1}}^\perp$.

Lemma \ref{lem: parabolic invariants} immediately implies the following.

\begin{cor}
\label{cor: parabolic invariants}
For all $\eta\vDash K$ we have
$$
\langle h_{\eta},\Frob_{q,t} H_*^{\BM}(X_{n,k})\rangle=\Hilb_{q,t} H^{\BM}_*(\pr_{\eta}(X_{n,k}))
$$
where $h_{\eta}$ is the complete symmetric function.
\end{cor}

Let
\begin{equation}\label{eq:ZDef}
Z \coloneqq \bigcup_{\substack{t'\in \overline{\mathcal{O}}_t\\ \ell(\JT(t'_{0}))>k}} \mathcal{O}_{t'}.
\end{equation}

\begin{lem}\label{lem:eval-in-sub}
    For all $g\mathbf{P}\in Y_{n,k}$, we have $\ev(\Ad(g^{-1})\gamma) \in \overline{\mathcal{O}}_t\setminus Z + \mathfrak{n}$. 
\end{lem}

\begin{proof}
    By construction, $\ev(\Ad(g^{-1})\gamma)\in \overline{\mathcal{O}}_t+\mathfrak{n}$. If $\ev(\Ad(g^{-1})\gamma)\in Z+\mathfrak{n}$, then $\ev(\Ad(g^{-1})\gamma) \in \mathcal{O}_{t'}$ for some $t'\in \overline{\mathcal{O}}_t$ such that $\ell(\JT(t'_{0}))>k$. Equivalently, converting this statement to lattices, this implies that $\gamma$ acts on $\Lambda_{0}/\Lambda_{K-n}$ by a nilpotent operator with Jordan type of length greater than $k$. It thus suffices to show that $\gamma|_{\Lambda_0/\epsilon\Lambda_0}$ has Jordan type of length at most $k$.

    Indeed, let $V = \bO\{\epsilon e_1,\epsilon\gamma e_1,\dots, \epsilon\gamma^{K-1}e_1\}$. In terms of the staircase diagram~\ref{fig:staircase-diagram}, $V$ is spanned by the vectors obtained by multiplying each of the vectors in the lower border of the staircase diagram by $\epsilon$. Observe that $V$ is closed under $\gamma$, and $\gamma$ acts on $\bO^K/V$ with Jordan type of length $k$. 
    
    We claim that for  $\Lambda_\bullet\in X_{n,k}$, then $\Lambda_0\cap V \subseteq \epsilon\Lambda_0$. Indeed, suppose there exists some $v\in \Lambda_0\cap V$ such that $v\notin \epsilon \Lambda_0$. Then $v\in \Lambda_0\setminus \epsilon \Lambda_0$, so $v$ has a leading term that's not in $V$, which means $v\notin V$, a contradiction.
    
    By the claim above, $\Lambda_0/\epsilon\Lambda_0$ is a subquotient of $\bO^K/V$ that's closed under $\gamma$:
    \[
        \bO^K/V \hookleftarrow \Lambda_0/(\Lambda_0\cap V) \twoheadrightarrow \Lambda_0/\epsilon\Lambda_0.
    \]
    Thus, we have the following set containment of Young diagrams 
    \begin{equation}\label{eq:JT Containment}
    (k(n-k+1),(k-1)(n-k+1),\dots, n-k+1)=\JT(\gamma|_{\bO^K/V}) \supseteq \JT(\gamma|_{\Lambda_0/\epsilon\Lambda_0}).
    \end{equation}
    Since the length of $\JT(\gamma|_{\bO^K/V})$ is $k$, then the length of $\JT(\gamma|_{\Lambda_0/\epsilon\Lambda_0})$ is at most $k$.
\end{proof}

\begin{prop}
    Every nonempty fiber of the map
    \[
    Y_{n,k}\to \widetilde{\Gr}
    \]
    sending $\Lambda_\bullet$ to $\Lambda_0$ is isomorphic to a $\Delta$-Springer fiber $Z_{n,\mu,k}$ for some partition $\mu$ (which vary over the fibers).
\end{prop}

\begin{proof}
    Let $\Lambda'\in \widetilde{\Gr}_K$ be a lattice whose fiber over $Y_{n,k}$ is nonempty. Then the fiber over $\Lambda$ is the space of partial flags of lattices $\Lambda_\bullet\in Y_{n,k}$ such that $\Lambda_0=\Lambda'$, which is isomorphic to the following space of partial flags $F_\bullet$ of type $(K-n,1^n)$ in the vector space $V=\Lambda_0/\epsilon\Lambda_0$,
    \begin{equation}\label{eq:PartialFlagsDescripFiber}
        \{F_\bullet\in \mathrm{Fl}_{(K-n,1^n)}(V)\mid \overline{\gamma}F_i\subseteq F_i,\, \JT(\overline{\gamma}|_{F_0/F_{K-n}})\leq (n-k)^{k-1}\}.
    \end{equation}
    where $\overline{\gamma} \coloneqq \gamma|_V$. Observe that the intersection with $C'$ in the definition of $Y_{n,k}$ does not affect the description of the fibers, only the $\Lambda'$ whose fibers are nonempty, since $C'$ is a preimage of a union of Schubert cells in $\widetilde{\Gr}$.
    
    By \eqref{eq:JT Containment}, we also have that $\JT(\gamma|_V)=\JT(\overline{\gamma})$ has length at most $k$. Since $\JT(\overline{\gamma}|_{F_0/F_{K-n}}) \leq (n-k)^{k-1}$, then $\mathrm{im}(\overline{\gamma}^{n-k}) = \overline{\gamma}^{n-k}F_0 \subseteq F_{K-n}$. Let $\mu = \JT(\overline{\gamma}|_{\mathrm{im}(\overline{\gamma}^{n-k})})$, which is the partition obtained by deleting the first $n-k$ columns of $\JT(\overline{\gamma})$. 
    
    We claim that \eqref{eq:PartialFlagsDescripFiber} is isomorphic to the $\Delta$-Springer fiber $Z_{n,\mu,k}$. Indeed, if the smallest row of $\JT(\overline{\gamma})$ is size $m$, then the truncated partition $\mu$ must have size $|\mu| \geq K - (k-1)(n-k)-m = n-m$. Thus, all rows of $\JT(\overline{\gamma})$ have size at least $n-|\mu|$ (note that $|\mu|$ here is playing the role of $k$ in the definition of the $\Delta$-Springer fiber, and $k$ is playing the role of $s$).
    
    Therefore, there exists $V'$ a $\overline{\gamma}$-invariant subspace of $V$ on which $\overline{\gamma}$ acts with Jordan type $\nu = (n-|\mu|)^{k} + \mu$. For example, $V'$ can be obtained by choosing a generalized eigenbasis for $\overline{\gamma}$ and defining $V'$ as a subspace spanned by an appropriate subset of these vectors. It can be checked that since $F_{K-n}$ contains the $|\mu|$-dimensional subspace $\mathrm{im}(\overline{\gamma}^{n-k})$, then $F_{K-n}$ is contained in $V'$. Thus, projection onto $V'$ yields an isomorphism between \eqref{eq:PartialFlagsDescripFiber} and $Z_{n,\mu,k}$ (defined accordingly using partial flags on $V'\cong \bC^{|\nu|}$).
\end{proof}

\begin{lem}\label{lem:rational-smoothness}
    The space $\overline{\mathcal{O}}_y$ is rationally smooth at all points in the preimage of $\ev(Y_{n,k})$ in $\overline{\mathcal{O}}_y$.
\end{lem}

\begin{proof}
    This follows from Lemma~\ref{lem:eval-in-sub} and \cite[Lemma 3.4]{GG}.
\end{proof}

\begin{thm}\label{thm:perp homology}
    There is a ``Springer'' action of $W^P\cong S_n$ on the Borel-Moore homology of $Y_{n,k}$, and we have the following identity
    \[
    \frac{1}{q^{\binom{k-1}{2}(n-k)}} s_{(n-k)^{k-1}}^\perp \Frob_{q,t}H_*^{\BM}(X_{n,k}) =  \Frob_{q,t} H_*^{\BM}(Y_{n,k}).
    \]
\end{thm}

\begin{proof}
    By \cite[1.10, Equation 6]{BM}, we have
    \begin{equation}\label{eq:IC_iso}
    R\xi^y_*\IC(\mathbb{Q}_{\mathcal{O}_y})[-2d_y]  \cong \bigoplus_x Rj_*^x \IC(\mathbb{Q}_{\mathcal{O}_x})[-2d_x]\otimes V_x^y
    \end{equation}
    where the sum runs over the adjoint orbits $\mathcal{O}_x$ in $\mathcal{N}$,  $j^y$ and $\xi$ are as in \eqref{eq:commutative-diag1}, and $j^x: \overline{\mathcal{O}}_x\hookrightarrow \mathfrak{g}$. Additionally, $d_x = \dim \pi^{-1}(x)$ and $d_y = \dim \xi^{-1}(y)$. Recall that $V_x^y\coloneqq \mathrm{Hom}_{W_L}(V_y,V_x)$, so that $V_x^y$ is the $W^P\cong S_n$-module such that 
    \[
    (V_x)^{V_y\text{-iso}} = V_x^y \otimes V_y
    \]
    and $V_x = H^{2d_x}(\mathcal{B}_x)=V_{\JT(x)}$ and 
    \[
    V_y = H^{2d_y}(\eta^{-1}(y))\cong H^{2d_y}(\mathcal{B}(L)_t) \cong H^{2d_y}(\mathcal{B}_{t_{0}})\cong V_{JT(t_0)}= V_{(n-k)^{k-1}},
    \]
    where all isomorphisms are as $W^P\cong S_n$-modules.

    Let $\rslocus$ be the preimage of $\ev\circ \pr(X_{n,k}) = \xi^y\circ \ev(X_{n,k})$ in $\mathcal{N}$, Restricting both sides of \eqref{eq:IC_iso} to $\rslocus$, by Lemma~\ref{lem:rational-smoothness} we have
    \[
    R\xi_*^y \mathbb{Q}_{\overline{\mathcal{O}}_y}|_\rslocus [-2d_y] \cong \bigoplus_x Rj^x_* \IC(\mathbb{Q}_{\mathcal{O}_x})|_\rslocus [-2d_x] \otimes V_x^y,
    \]
    and then tensoring both sides by $V_y$ we have
    \begin{equation}\label{eq:decomp1}
    (R\xi_*^y \mathbb{Q}_{\overline{\mathcal{O}}_y})|_\rslocus \otimes V_y [-2d_y] \cong \bigoplus_x Rj^x_* \IC(\mathbb{Q}_{\mathcal{O}_x})|_\rslocus[-2d_x]\otimes (V_x)^{V_y\text{-iso}}.
    \end{equation}
    Next, by the Decomposition Theorem applied to $\pi$ (restricted to $\rslocus$) we have
    \[
    \bigoplus_x Rj^x_* \IC(\mathbb{Q}_{\mathcal{O}_x})|_\rslocus[-2d_x]\otimes V_x \cong R\pi_* \mathbb{Q}_{\widetilde{\mathcal{N}}}|_\rslocus
    \]
    where $W\cong S_K$ acts by the Springer representation on $V_x$ and on the sheaf on the right-hand side. Taking the $V_y$-isotypic components of each side, we have
    \begin{equation}\label{eq:decomp2}
    \bigoplus_x Rj^x_* \IC(\mathbb{Q}_{\mathcal{O}_x})|_\rslocus[-2d_x]\otimes (V_x)^{V_y\text{-iso}} \cong (R\pi_* \mathbb{Q}_{\widetilde{\mathcal{N}}}|_\rslocus)^{V_y\text{-iso}}
    \end{equation}
    Combining \eqref{eq:decomp1} and \eqref{eq:decomp2},
    \[
    (R\xi_*^y \mathbb{Q}_{\overline{\mathcal{O}}_y})|_\rslocus \otimes V_y [-2d_y] \cong (R\pi_* \mathbb{Q}_{\widetilde{\mathcal{N}}}|_\rslocus)^{V_y\text{-iso}}.
    \]

    Since $\rslocus$ is a union of left $G$ orbits, this isomorphism descends to an isomorphism of objects in $D_G(\rslocus) = D([G\backslash \rslocus])$,
    \[
    (R\hat\xi_*^y \mathbb{Q}_{[G\backslash\overline{\mathcal{O}}_y]})|_\rslocus \otimes V_y [-2d_y] \cong (R\hat\pi_* \mathbb{Q}_{[G\backslash\widetilde{\mathcal{N}}]}|_\rslocus)^{V_y\text{-iso}}.
    \]
    
    Applying $\ev^*$ to both sides and then using proper base change, we have
    \[
    (R\pr_* \mathbb{Q}_{Y_{n,k}})\otimes V_y[-2d_y] \cong (R\pr_* \mathbb{Q}_{\Sp_\gamma})^{V_y\text{-iso}}
    \]
    Now applying Verdier duality to both sides,
    \[
    (R\pr_* \mathbb{D}_{Y_{n,k}})\otimes V_y^*[2d_y] \cong (R\pr_* \mathbb{D}_{\Sp_\gamma})^{V_y\text{-iso}},
    \]
    where on the right-hand side Verdier duality commutes with taking the $V_y$-isotypic component since $V_y^*\cong V_y$. 
    
    Recall that $H^{\BM}_i(X;\bQ)\cong H^{-i}(X,\mathbb{D}_X)$.
    Pushing forward both sides to a point, we get
    \[
    H^{\BM}_{i-2d_y}(Y_{n,k};\bQ)\otimes V_y \cong H^{\BM}_{i}(\Sp_\gamma;\bQ)^{V_y\text{-iso}},
    \]
    for all $i$, where the isomorphism is as a graded $W^P\cong S_n$-module, where the $W^P$ action on the left is inherited from the $W^P$ action on $R\xi^y_* \mathrm{IC}(\bQ_{\mathcal{O}_y})$, and the $W^P$ action on the right-hand side is inherited from the usual Springer action. Observing that $d_y = \binom{k-1}{2}(n-k)$ and applying Lemma \ref{lem:Skewing}, the result follows.
\end{proof}

\section{Affine paving of $X_{n,k}$}
\label{sec: X dimensions}

\subsection{Affine paving}

Let $\omega^{-1}$ be the permutation obtained from a parking function $\pi$ as in Section \ref{sec:P-w}, so that $\omega^{-1}(i)$ is the rank of the cell of parking function $\pi$ labeled by $i$. Letting $C_{\omega^{-1}}$ be the Schubert cell $I^{-}\omega^{-1}I^-\subseteq \mathbf{G}/I^-$, then 
  \[
  \dim C_{\omega^{-1}} = \inv(\omega^{-1}) = |\{(\alpha,\beta)\mid 1\leq \omega(\alpha)\leq K, \omega(\beta)<\omega(\alpha),\alpha < \beta\}|.
  \]
  
We will need some explicit coordinates on the Schubert cell defined as follows.
Let $\Lambda_\bullet\in C_{\omega^{-1}}$ with $\Lambda_0\supseteq \Lambda_1\supseteq \Lambda_2\supseteq \cdots \supseteq \Lambda_K =\epsilon \Lambda_0$, and let $f_\alpha$ be the canonical generator of $\Lambda_{\omega(\alpha)-1}/\Lambda_{\omega(\alpha)}$ for each $1\leq \omega(\alpha)\leq K$, with expansion \eqref{eq: canonical}
    \[
    f_\alpha = e_\alpha + \sum_{\stackrel{\beta>\alpha}{ \omega(\beta)<\omega(\alpha)}} \lambda_\beta^\alpha e_\beta
    \]
    where we may extend the definition of $\lambda_\beta^\alpha$ to any $\alpha$ and $\beta$ by declaring $\lambda^\alpha_\beta = \lambda^{\alpha + K}_{\beta+K}$ and $\lambda^\alpha_\beta = 0$ if $\alpha > \beta$ or $\omega(\beta)>\omega(\alpha)$.

\begin{defn}
For $\alpha = mK+qk+r$ with $1\leq r\leq k$ and $0\leq q\leq n-k$, define 
    \begin{align*}
    \phi_0(\alpha) &\coloneqq mK+qk\\
    \phi_\infty(\alpha) &\coloneqq q+(n-k+1)(r-1).
    \end{align*}
    Furthermore, we define 
    \begin{align*}
    \phi_0(\lambda^{\alpha}_{\beta})&=\phi_0(\beta)-\phi_0(\alpha),\\
    \phi_\infty(\lambda^{\alpha}_{\beta})&=\phi_\infty(\beta)-\phi_\infty(\alpha).
    \end{align*}
\end{defn}
    
    Clearly, if $\alpha<\beta$ then $\phi_0(\lambda^{\alpha}_{\beta}),\phi_\infty(\lambda^{\alpha}_{\beta})\ge 0$ and at least one of these is nonzero.

\begin{defn}\label{def: Var Order}
Define $\lambda^\alpha_\beta\succ \lambda^{\alpha'}_{\beta'}$ if 
    \begin{itemize}
        \item $\phi_0(\lambda^\alpha_\beta) > \phi_0(\lambda^{\alpha'}_{\beta'})$, or
        \item $\phi_0(\lambda^\alpha_\beta) = \phi_0(\lambda^{\alpha'}_{\beta'})$ and $\phi_\infty(\lambda^\alpha_\beta) > \phi_\infty(\lambda^{\alpha'}_{\beta'})$.
    \end{itemize}
\end{defn}

Observe that this ordering respects the relation $\lambda^\alpha_\beta = \lambda^{\alpha+K}_{\beta+K}$.  Now, assume $N\gg 0$ and $X_{n,k}=X_{n,k,N}$ as above.
    
\begin{lem}
\label{lem: equations cell}
The intersection $X_{n,k}\cap C_{\omega^{-1}}$  is cut out in  the Schubert cell $C_{\omega^{-1}}$ by the equations:
\begin{equation}\label{eq:ElimEqn2 new}
 \lambda^\alpha_\beta -\lambda^{\gamma(\alpha)}_{\gamma(\beta)}=
 \sum_{d\geq 2}(-1)^d\sum_{\beta_1,\ldots,\beta_{d-1}}\left(\lambda^{\alpha}_{\beta_1}-\lambda^{\gamma(\alpha)}_{\gamma(\beta_1)}\right)\lambda^{\gamma(\beta_1)}_{\gamma(\beta_2)}\cdots \lambda^{\gamma(\beta_{d-1})}_{\gamma(\beta)}
 \end{equation}
for all $\alpha,\beta$ such that $1\le \omega(\alpha)\le K$, $\omega(\gamma(\beta))<\omega(\alpha)$.  The sum is over $\beta_1,\ldots,\beta_{d-1}$ such that
$$
\gamma(\beta_1)<\gamma(\beta_2)<\cdots<\gamma(\beta_{d-1})<\gamma(\beta).
$$
Furthermore, all coefficients $\lambda^{\alpha'}_{\beta'}$ appearing in the nonlinear terms of  \eqref{eq:ElimEqn2 new} are strictly less than $\lambda^{\alpha}_{\beta}$ and $\lambda^{\gamma(\alpha)}_{\gamma(\beta)}$  in our ordering.
\end{lem}

\begin{proof}
First, recall from \eqref{eq: canonical} that if $\Lambda_\bullet \in C_{\omega^{-1}}$ then $\Lambda_{i-1}/\Lambda_i$ is spanned by the canonical generator $f_{\omega^{-1}(i)}$ for all $i$.  We can alternatively restate this as the fact that for any $\alpha$, the quotient $\Lambda_{\omega(\alpha)-1}/\Lambda_{\omega(\alpha)}$ is spanned by $$f_\alpha=e_\alpha+\sum_{\substack{\beta>\alpha\\ \omega(\beta)<\omega(\alpha)}} \lambda_\beta^\alpha e_\beta.$$ Thus $\Lambda_{\omega(\alpha)-1}$ is spanned by $$\{f_{\beta'}\mid \omega(\beta')\ge \omega(\alpha)\}=\{e_{\beta'}+\sum_{\beta>\beta'} \lambda_\beta^{\beta'} e_\beta\mid \omega(\beta')\ge \omega(\alpha)\}.$$

Now, note that by the definition of $X_{n,k}$, we have that $\Lambda_\bullet\in X_{n,k}$ if and only if the vector
    \begin{equation}\label{eq:ElimEqn1 new}
    \gamma f_\alpha - f_{\gamma(\alpha)} = \sum_{\beta > \alpha}\lambda^\alpha_\beta e_{\gamma(\beta)} - \sum_{\rho > \gamma(\alpha)} \lambda_\rho^{\gamma(\alpha)} e_\rho = 
    \sum_{\beta > \alpha}\lambda^\alpha_\beta e_{\gamma(\beta)} - \sum_{\gamma(\beta) > \gamma(\alpha)} \lambda_{\gamma(\beta)}^{\gamma(\alpha)} e_{\gamma(\beta)}
    \end{equation}
    is in $\Lambda_{\omega(\alpha)-1}$ for all $\alpha$ such that $1\le \omega(\alpha)\le K$.  If $\omega(\gamma(\beta))>\omega(\alpha)$, then we can eliminate the $e_{\gamma(\beta)}$ term using the corresponding $f_{\gamma(\beta)}$ generator; suppose we do so over all such $\beta$ in increasing order of the value of $\gamma(\beta)$.  Now, fix an index $\beta$ such that $\omega(\gamma(\beta))<\omega(\alpha)$.  We consider the coefficient of $e_{\gamma(\beta)}$ after all such eliminations as outlined above.

    Note that its original coefficient was $\lambda_\beta^\alpha-\lambda_{\gamma(\beta)}^{\gamma{\alpha}}$.  This coefficient can be changed under our elimination process by any $\gamma(\beta')<\gamma(\beta)$, which in turn can be affected by previous values, so in particular we consider the maximal chain of elements $\beta_1,\ldots,\beta_m$ such that $\gamma(\beta_1)<\cdots <\gamma(\beta_{m})<\gamma(\beta)$ such that for all $i$ we have $\omega(\gamma(\beta_i))>\omega(\alpha)$.  We eliminate $e_{\gamma(\beta_1)}$ first by subtracting $$(\lambda_{\beta_1}^\alpha-\lambda_{\gamma(\beta_1)}^{\gamma(\alpha)})f_{\gamma(\beta_1)}=\sum_{\beta'>\gamma(\beta_1)}(\lambda_{\beta_1}^\alpha-\lambda_{\gamma(\beta_1)}^{\gamma(\alpha)}) \lambda_{\beta'}^{\gamma(\beta_1)}e_{\beta'}.$$  Looking at the terms where $\beta'=\beta_2,\beta_3,\ldots$, we notice that $e_{\gamma(\beta_2)}$ now has a coefficient of \begin{equation}\label{eq:beta-2}(\lambda_{\beta_2}^\alpha-\lambda_{\gamma(\beta_2)}^{\gamma(\alpha)})-(\lambda_{\beta_1}^{\alpha}-\lambda_{\gamma(\beta_1)}^{\gamma(\alpha)})\lambda_{\gamma(\beta_2)}^{\gamma(\beta_1)},\end{equation} and similarly $e_{\gamma(\beta_3)}$ now has a coefficient of $$(\lambda_{\beta_3}^\alpha-\lambda_{\gamma(\beta_3)}^{\gamma(\alpha)})-(\lambda_{\beta_1}^{\alpha}-\lambda_{\gamma(\beta_1)}^{\gamma(\alpha)})\lambda_{\gamma(\beta_3)}^{\gamma(\beta_1)}.$$  We next eliminate $e_{\gamma(\beta_2)}$ by subtracting the coefficient \eqref{eq:beta-2} times $f_{\gamma(\beta_2)}$ and find that the coefficient of $e_{\gamma(\beta_3)}$ is now $$(\lambda_{\beta_3}^\alpha-\lambda_{\gamma(\beta_3)}^{\gamma(\alpha)})-\lambda_{\gamma(\beta_1)}^{\gamma(\alpha)})\lambda_{\gamma(\beta_3)}^{\gamma(\beta_1)}-(\lambda_{\beta_2}^\alpha-\lambda_{\gamma(\beta_2)}^{\gamma(\alpha)})\lambda_{\gamma(\beta_3)}^{\gamma(\beta_2)}+(\lambda_{\beta_1}^{\alpha}-\lambda_{\gamma(\beta_1)}^{\gamma(\alpha)})\lambda_{\gamma(\beta_2)}^{\gamma(\beta_1)}\lambda_{\gamma(\beta_3)}^{\gamma{\beta_2}}$$ which is an alternating sum in a similar form to \eqref{eq:ElimEqn2 new}, summing over chains of $\beta_i$'s between $\beta_1$ and $\beta_3$.
    Continuing in this manner and setting the final coefficient of $e_{\gamma(\beta)}$ to be $0$, we obtain the desired equation.

    
    For the second claim, note that for $N\gg 0$, the quantities $\phi_0(\lambda^{\alpha}_{\beta})$ and $\phi_\infty(\lambda^{\alpha}_{\beta})$ are the weights of the $\mathbb{C}^*\times \mathbb{C}^*$ action \eqref{eq: theta stable} on the Schubert cell $C_{\omega^{-1}}$. By Lemma \ref{lem: stable torus action} this action preserves the subvariety $X_{n,k}=X_{n,k,N}$ for $N\gg 0$ and hence the equations \eqref{eq:ElimEqn2 new} are homogeneous with respect to both $\phi_0$ and $\phi_{\infty}$. Suppose a monomial $\lambda^{\alpha_1}_{\beta_1}\cdots \lambda^{\alpha_k}_{\beta_k}$
 appears in the right hand side of \eqref{eq:ElimEqn2 new} with a nonzero coefficient, and $k>1$. Then 
 $$
\sum_{i=1}^{k}\phi_0\left(\lambda^{\alpha_i}_{\beta_i}\right)=\phi_0\left(\lambda^{\alpha}_{\beta}\right),\ \sum_{i=1}^{k}\phi_\infty\left(\lambda^{\alpha_i}_{\beta_i}\right)=\phi_\infty\left(\lambda^{\alpha}_{\beta}\right),
 $$
 hence
$$
\phi_0\left(\lambda^{\alpha_i}_{\beta_i}\right)\le \phi_0\left(\lambda^{\alpha}_{\beta}\right),\ \phi_\infty\left(\lambda^{\alpha_i}_{\beta_i}\right)\le \phi_\infty\left(\lambda^{\alpha}_{\beta}\right).
$$
If we assume that both of these inequalities are equalities for some $i$, then for all $j\neq i$ we must have
$\phi_0\left(\lambda^{\alpha_i}_{\beta_i}\right)=\phi_\infty\left(\lambda^{\alpha_i}_{\beta_i}\right)=0$ which is not possible. We have a contradiction, so $\lambda^{\alpha_i}_{\beta_i}\prec \lambda^{\alpha}_{\beta}$.
\end{proof}    

By Corollary \ref{cor: cells Y} the intersection $X_{n,k}\cap C_{\omega^{-1}}$ is nonempty if and only if $\omega^{-1}$ is $\gamma$-restricted. By Lemma \ref{lem: pf bijection} there exists a $(K,k)$ parking function $\pi$ such that $\omega^{-1}=\omega_{\pi}$.

\begin{thm}
\label{thm: dim cell}
For a $(K,k)$ parking function $\pi$, the intersection $X_{n,k}\cap C_{\omega_\pi}$ is isomorphic to an affine space of dimension $\delta_{K,k}-\dinv(\pi)$ (where $\delta_{K,k}$ is defined in Definition~\ref{def:delta-num}).
\end{thm}

\begin{proof}
For notational convenience, set $\omega^{-1} = \omega_\pi$. 
By Lemma \ref{lem: equations cell} the intersection $X_{n,k}\cap C_{\omega^{-1}}$ is cut out by the equations \eqref{eq:ElimEqn2 new}, where $\omega(\alpha) > \omega(\gamma(\beta))$.  We would like to better control the linear terms of these equations. Recall that $1\le \omega(\alpha)\le K$ and $\omega^{-1}$ is $\gamma$-restricted.

The term $\lambda_{\beta}^{\alpha}$ appears in \eqref{eq:ElimEqn2 new} if $\alpha<\beta$, since we are already given that $\omega(\alpha)>\omega(\gamma(\beta))$, which implies $\omega(\alpha)>\omega(\beta)$ by the fact that $\omega^{-1}$ is $\gamma$-restricted.

The term $\lambda_{\gamma(\beta)}^{\gamma(\alpha)}$ appears in \eqref{eq:ElimEqn2 new} if $\gamma(\alpha)<\gamma(\beta)$. Indeed, since we are already given that $\omega(\alpha)>\omega(\gamma(\beta))$, by Lemma \ref{lem:A-properties} this implies  $\alpha<\beta$ and $\omega(\gamma(\alpha))>\omega(\gamma(\beta))$. In particular, in this case both $\lambda_{\beta}^{\alpha}$ and $\lambda_{\gamma(\beta)}^{\gamma(\alpha)}$ appear in the equation since they are not identically zero.

Assume that $\omega(\alpha)>\omega(\gamma(\beta))$ but neither of the linear terms appear in the equation \eqref{eq:ElimEqn2 new}. 
We claim that in this case nonlinear terms cannot appear as well, and the equation is vacuous. 
Indeed, we have two cases.

\textbf{Case 1.} Suppose $\lambda^{\gamma(\alpha)}_{\gamma(\beta_1)}\lambda^{\gamma(\beta_1)}_{\gamma(\beta_2)}\cdots \lambda^{\gamma(\beta_{d-1})}_{\gamma(\beta)}$ appears, meaning that it is not identically zero. Then $\gamma(\alpha)<\gamma(\beta_1)<\gamma(\beta_2)<\cdots <\gamma(\beta)$, so $\gamma(\alpha)<\gamma(\beta)$, and both linear terms $\lambda^\alpha_\beta$ and $\lambda^{\gamma(\alpha)}_{\gamma(\beta)}$ appear.

\textbf{Case 2.} Suppose $\lambda^{\alpha}_{\beta_1}\lambda^{\gamma(\beta_1)}_{\gamma(\beta_2)}\cdots \lambda^{\gamma(\beta_{d-1})}_{\gamma(\beta)}$ appears. Then $\alpha<\beta_1$ and $\gamma(\beta_1)<\gamma(\beta_2)<\cdots <\gamma(\beta)$. If $\beta\neq mK$ then applying Lemma \ref{lem: gamma almost increasing} to $\gamma(\beta_1)<\gamma(\beta)$ we get $\beta_1<\beta$, so $\alpha<\beta$ and  the linear term $\lambda_{\beta}^{\alpha}$ appears.
If $\beta=mK$, then we write
$$
k<\alpha+k<\beta_1+k\le \gamma(\beta_1)<\gamma(\beta)\ 
\mathrm{while}\ 
\omega(\gamma(\beta))<\omega(\alpha)\le K. 
$$
By Corollary \ref{cor: beta mK} we get a contradiction.

Therefore, by Lemma \ref{lem: equations cell} each nontrivial equation presents $\lambda^{\alpha}_{\beta} - \lambda^{\gamma(\alpha)}_{\gamma(\beta)}$ as a polynomial in smaller variables in our ordering. Observe that $\lambda^{\alpha}_{\beta}$ and $\lambda^{\gamma(\alpha)}_{\gamma(\beta)}$ have the same $\phi_0$ and $\phi_\infty$ weight, and hence are incomparable under $\prec$. Therefore, we may refine the ordering on the variables so that $\lambda^{\alpha}_{\beta} \succ \lambda^{\gamma(\alpha)}_{\gamma(\beta)}$ for $\alpha \neq mK$ for any $m$. 

Suppose \eqref{eq:ElimEqn2 new} appears as an equation when $\alpha = mK$ for some $m$, and suppose $\lambda_{\gamma(\beta)}^{\gamma(\alpha)}$ appears (so that both linear terms appear by the reasoning above). First, suppose $\beta\neq m'K$ for any $m'$. Since $\gamma(\alpha) < \gamma(\beta)$, then $\alpha + NK+1 < \gamma(\beta) \leq \beta + k + 1$. Since $N\geq k$, this implies $\beta > \alpha + (k-1)K + (K-k+1) > (k-1)K + (k-1)(n-k+1)$, thus $\gamma(\beta) > kK$. Since $\omega^{-1}$ is $\gamma$-restricted, then by Lemma~\ref{lem: A Wild Conductor Appears}, $\omega(\gamma\beta) > K$, contradicting $K\geq \omega(\alpha) > \omega(\gamma(\beta))$. If $\beta = m'K$, then $mK<m'K$ and $\omega(mK)>\omega(m'K)$ cannot simultaneously hold, a contradiction.

Thus, the equation for $\alpha=mK$ does not contain the $\lambda^{\gamma(\alpha)}_{\gamma(\beta)}$ linear term.



Therefore, we can eliminate the variables one by one in order in terms of variables which are smaller under $\prec$ and use remaining variables as free coordinates. 

The number of equations equals $| A_{\omega}|$ where $A_{\omega}$ is given by \eqref{eq: def A}, so 
$$
X_{n,k}\cap C_{\omega^{-1}}\simeq \mathbb{C}^{\inv(\omega^{-1})-|A_{\omega}|}=\mathbb{C}^{\delta_{K,k}-\dinv(\pi)}.
$$
The last equation follows from Lemma \ref{lem: A size}.
\end{proof}

\subsection{Proof of Theorem~\ref{thm:GradedFrob}}

In this section we prove Theorem \ref{thm:GradedFrob}. We recall some definitions and results from \cite{GGG1}. 

First, we have a degree $K$ symmetric function $E_{K,k}\cdot 1$. Here $E_{K,k}$ is an operator in the {\em elliptic Hall algebra} acting on symmetric functions, and we apply it to 1. We refer to \cite[Section 3]{GGG1} for all definitions and references. For the purposes of this paper, the reader can safely use the right hand side of the equation \eqref{eq: rational shuffle} as the definition of $E_{K,k}\cdot 1$.

A $(K,k)$ word parking function is  a labeling
of the vertical runs of a $(K,k)$ Dyck path by positive integers such that the labeling strictly increases up
each vertical run (but letters may repeat between columns), we denote the set of such functions by $\WPF_{K,k}$. Furthermore, for $\eta\vDash K$ we define 
$\WPF_{K,k,\eta}$ as the set of word parking functions with labels from $1$ to $\ell(\eta)$ and content $\eta$.
We write $\wWPF_{K,k}$ and $\wWPF_{K,k,\eta}$ for the same labelings except allowing the labeling to weakly increase up each vertical run.

The $\dinv$ statistics on word parking functions is defined exactly as in Definition \ref{def: dinv}, note that an attacking pair with equal letters $a=b$ does not contribute to $\tdinv$ or to $\dinv$. We will need another version (denoted by $\dinv'$) where in fact we assume that all attacking pairs with $a=b$ contribute to $\dinv$.

The following result is known as the Rational Shuffle Conjecture, proposed in \cite{RationalShuffle} and proved in \cite{Mellit}.  Note that in this section, we use $\omega$ both for the involution on symmetric functions that sends a Schur function $s_\lambda$ to the transpose Schur function $s_{\lambda'}$, and for an affine permutation, with the understanding that the two meanings can be easily distinguished by context.

We also recall that in general, applying the involution $\omega$ to an LLT polynomial has the combinatorial effect of allowing vs disallowing ties and interchanging $\dinv$ with $\dinv'$.

\begin{thm}[\cite{Mellit}]
\label{thm: rational shuffle}
We have
\begin{align}
\label{eq: rational shuffle}
E_{K,k}\cdot 1&=\sum_{\pi\in \WPF_{K,k}}q^{\area(\pi)}t^{\dinv(\pi)}x^{\pi},\\
\omega (E_{K,k}\cdot 1)&=\sum_{\pi\in \wWPF_{K,k}}q^{\area(\pi)}t^{\dinv'(\pi)}x^{\pi}.\label{eq: rational shuffle omega}
\end{align}

Equivalently, for all compositions $\eta\vDash K$ we have
\begin{equation}
\label{eq: E paired with complete}
\langle h_{\eta},\omega (E_{K,k}\cdot 1)\rangle=\sum_{\pi\in \wWPF_{K,k,\eta}} t^{\area(\pi)}q^{\dinv'(\pi)}
\end{equation}
where $h_{\eta}$ is the complete symmetric function.
\end{thm}
Note that here $\omega$ is the involution on symmetric functions, which is not to be confused with the affine permutation $\omega\in \widetilde{S_K}$. Further note that \eqref{eq: rational shuffle omega} follows easily from \eqref{eq: rational shuffle} by the expansion of the right-hand side of \eqref{eq: rational shuffle} into LLT polynomials and how $\omega$ acts on LLT polynomials, see for example Section 2 in \cite{HHLRU} on the technique of \emph{superization}, specifically Corollary 2.4.3 in \emph{loc. cit}.

By Lemma \ref{lem: pf bijection} there is a bijection between $\PF_{K,k}$ and the set of $\gamma$-restricted affine permutations. We will need an analogue of this bijection for $\wWPF_{K,k,\eta}$. Let $\eta\vDash K$, and $\pi\in \wWPF_{K,k,\eta}$. We define its standardization $\std(\pi)\in \PF_{K,k}$ as follows: this is a parking function such that all the boxes in $\pi$ labeled by $j$ are replaced by the labels in the $j$-th block of $\eta$. Within each block, we order the labels according to their ranks.

\begin{lem}
\label{lem: wpf bijection}
The map $\wWPF_{K,k,\eta}\to \PF_{K,k}$ defined by $\pi\mapsto \std(\pi)$ has the following properties:
\begin{enumerate}
\item $\dinv'(\pi)=\dinv(\std(\pi))$.
\item The affine permutation $\omega_{\std(\pi)}$ is the minimal length representative in $\widetilde{S_K}/S_{\eta}$.
\item The map $\pi\to \omega_{\std(\pi)}$ is a bijection between $\wWPF_{K,k,\eta}$ and the set of $\gamma$-restricted minimal length representatives in $\widetilde{S_K}/S_{\eta}$.
\end{enumerate}
\end{lem}

\begin{proof}
Part (a) is clear from definitions. For parts (b) and (c), we recall that minimizing the length is equivalent to minimizing $\inv(\omega)$. A word parking function $\pi\in \wWPF_{K,k,\eta}$ corresponds to a collection of parking functions $\widetilde{\pi}\in \PF_{K,k}$ where the labels in different blocks of $\eta$ have the same relative order as in $\pi$, but the labels within one block are permuted arbitrarily. For all such $\widetilde{\pi}$ the corresponding affine permutations $\omega_{\widetilde{\pi}}$ are in the same coset in $\widetilde{S}_K/S_{\eta}$. It remains to notice that $\inv(\omega_{\widetilde{\pi}})$ is minimal if the labels in each block are ordered by their ranks, which is precisely our definition of $\std(\pi)$.

\end{proof}

Next, we state the main result of \cite{GGG1} which relates the Rational Shuffle Theorem to the Delta Theorem. 

\begin{thm}\cite{GGG1}
\label{thm:PerpSymmFunc}
We have 
$$
s^{\perp}_{(k-1)^{n-k}}(E_{K,k}\cdot 1)=\Delta'_{e_{k-1}}e_n,\quad  
s^{\perp}_{(n-k)^{k-1}}\omega (E_{K,k}\cdot 1)=\omega \Delta'_{e_{k-1}}e_n.
$$
\end{thm}

Now we are ready to prove Theorem~\ref{thm:GradedFrob}.

\begin{proof}[Proof of Theorem~\ref{thm:GradedFrob}]
    By Theorem \ref{thm: dim cell}, the space $X_{n,k}$ admits an affine paving with cells 
    \[
    X_{n,k}\cap C_{\omega^{-1}} \simeq \mathbb{C}^{\delta_{K,k}-\dinv(\pi(\omega))}.
    \]
    Here $\pi(\omega)$ is the $(K,k)$ parking function corresponding to the $\gamma$-restricted permutation $\omega^{-1}$. Therefore
    $$
    \Hilb_{q,t}H_*^{\BM}(X_{n,k})=\sum_{\pi\in \PF_{K,k}}q^{\area(\pi)}t^{\dinv(\pi)}.
    $$
    Let $\eta\vDash K$ and let $\pr_\eta : \Sp_\gamma\to \Sp_{\gamma,\eta}$ be the projection map \eqref{eq: pr eta}. 
    The projection $\pr_{\eta}(X_{n,k})\subset \aFl_{\eta}$ also admits an affine paving which can be constructed as follows. By Proposition \ref{prop: schubert cells in partial flags} the Schubert cells in $\aFl_{\eta}$ can be written as $\pr_{\eta}(C_{\omega^{-1}})$ where $\omega^{-1}$ is the minimal length coset representative in $\widetilde{S}_K/S_{\eta}$. Furthermore, $\pr_{\eta}:C_{\omega^{-1}}\to \pr_{\eta}(C_{\omega^{-1}})$ is an isomorphism, so
    $$
    \pr_{\eta}(X_{n,k})\cap \pr_{\eta}(C_{\omega^{-1}})=\pr_{\eta}(X_{n,k}\cap C_{\omega^{-1}})\simeq X_{n,k}\cap C_{\omega^{-1}}\simeq \mathbb{C}^{\delta_{K,k}-\dinv(\pi(\omega))}.
    $$
    By Lemma \ref{lem: wpf bijection} we get
    $$
    \Hilb_{q,t}H_*^{\BM}(\pr_{\eta}(X_{n,k}))=\sum_{\pi\in \WPF_{K,k,\eta}}q^{\area(\pi)}t^{\delta_{K,k}-\dinv'(\pi)}
    $$
    Now we combine these results with Corollary \ref{cor: parabolic invariants}  and write
    \begin{align*}
    \left\langle h_\eta, \Frob_{q,t}H_*^{\BM}(X_{n,k})\right\rangle &= \Hilb_{q,t}H_*^{\BM}(X_{n,k})^{S_\eta}\\
    &= \Hilb_{q,t}H_*^{\BM}(\pr_\eta(X_{n,k}))\\
    &=\sum_{\pi\in \WPF_{K,k,
    \eta}} t^{\area(\pi)}q^{\delta_{K,k}-\dinv'(\pi)}\\
    &= \left\langle h_\eta, \rev_q \omega (E_{k,K}\cdot 1)\right\rangle
    \end{align*}
The last equation follows from \eqref{eq: E paired with complete}.    
Thus, 
\[
\Frob_{q,t}H_*^{\BM}(X_{n,k}) = \rev_q \omega E_{k,K}\cdot 1,
\]
as desired. Applying $s_{(n-k)^{k-1}}^\perp$ to both sides, multiplying both sides by $q^{-\binom{k}{2}(n-k)}$, and applying Theorem~\ref{thm:perp homology} and then Theorem~\ref{thm:PerpSymmFunc} we get
\begin{align*}
\Frob_{q,t}H_*^{\BM}(Y_{n,k}) &= \frac{1}{q^{\binom{k-1}{2}(n-k)}} s_{(n-k)^{k-1}}^\perp \rev_q \omega (E_{K,k}\cdot 1)\\
&= q^{\delta_{K,k}-\binom{k-1}{2}(n-k)} s_{(n-k)^{k-1}}^\perp  (\omega E_{K,k}\cdot 1)(q^{-1},t)\\
&= q^{\delta_{K,k}-\binom{k-1}{2}(n-k)} \omega \Delta'_{e_{k-1}}e_n(q^{-1},t)\\
&= q^{\binom{k}{2} + (k-1)(n-k)} \omega \Delta'_{e_{k-1}}e_n(q^{-1},t)\\
&= \rev_q \omega\Delta'_{e_{k-1}}e_n
\end{align*}
and we are done.
\end{proof}

\bibliographystyle{plain}
\bibliography{refs.bib}

\end{document}